\documentclass[nonacm, screen]{acmart} 
\AtBeginDocument{%
  }

\usepackage{bussproofs}

\usepackage{tikz} 
\usepackage{tikz-cd} 
\usetikzlibrary{cd} 

\usepackage[shortlabels]{enumitem}

\usepackage[normalem]{ulem} 

\usepackage{stmaryrd}
\usepackage{proof} 

\usepackage{refcount}

\setcopyright{acmlicensed}
\copyrightyear{2026}
\acmYear{2026}
\acmDOI{XXXXXXX.XXXXXXX}
\acmConference[LICS 2026]{Make sure to enter the correct
  conference title from your rights confirmation email}{July 20--23,
  2026}{Lisbon, Portugal}
\acmISBN{978-1-4503-XXXX-X/2018/06}

\begin{document}

\title{The logic of bunched implications is undecidable}

\author{Nick Galatos}
\affiliation{%
  \institution{University of Denver}
  \city{Denver}
  \state{Colorado}
  \country{USA}}
\email{ngalatos@du.edu}

\author{Peter Jipsen}
\affiliation{%
  \institution{Chapman University}
  \city{Orange}
  \state{California}
  \country{USA}
}
\email{jipsen@chapman.edu}

\author{Søren Brinck Knudstorp}
\affiliation{%
  \institution{ILLC and Philosophy, University of Amsterdam}
  \city{Amsterdam}
  \country{the Netherlands}
}
\email{s.b.knudstorp@uva.nl}
\orcid{0009-0008-9835-4195}

\author{Revantha Ramanayake}
\affiliation{%
 \institution{University of Groningen}
 \city{Groningen}
 \country{the Netherlands}}
\email{d.r.s.ramanayake@rug.nl}
\orcid{0000-0002-7940-9065}


\begin{abstract}
The logic of bunched implications (BI), introduced by O’Hearn and Pym (1999), has attracted significant attention due to its
elegant proof calculus, varied semantics, and close connections to the propositional fragment of separation logic. We show here that provability in BI is undecidable by encoding Wang tilings into its  ternary relational semantics. Equivalently, this yields the undecidability of the equational theory of BI-algebras.

Our result is much more general, applying to the $\{\land, \lor, \neg, \mimp\}$-fragment of stronger and weaker logics: the negation simply needs to be disjointive, and 
the multiplicative conjunction need not be commutative (then $\mimp$ splits into two divisions $\ld, \rd$). Consequently, our result covers an interval that includes BI, the non-commutative logic GBI, and Boolean BI (BBI), the latter already known to be undecidable.

This result contrasts with a long-standing expectation that BI might be decidable.
We also identify the gaps in the publications claiming decidability.
\end{abstract}




\theoremstyle{acmdefinition} 
\newtheorem{remark}[theorem]{Remark}

\newcounter{savedcounter}
\newcounter{savedsection}

\newcommand{\dom}{\mathrm{dom}} 
\newcommand{\Log}{\mathrm{Log}} 
\newcommand{\Z}{\mathbb{Z}}
\newcommand{\Q}{\mathbb{Q}}
\newcommand\sepimp{\mathrel{-\mkern-3mu\ast}} 

\newcommand{\LBIold}{\text{LBI}_{0}}
\newcommand{\DC}{\Rightarrow}
\newcommand{\IL}{\text{Ip}}
\newcommand{\MILL}{\text{MILL}}
\newcommand{\mimp}{\sepimp}
\newcommand{\imp}{\rightarrow}
\newcommand{\depth}{\mathsf{d}}

\newcommand{\Ra}{\Rightarrow}
\newcommand{\BI}{\text{BI}}
\newcommand{\config}[2]{#1 #2}
\newcommand{\qRorR}{q_f\oplus 1\oplus q_f}
\newcommand{\qore}{q\oplus\epsilon\oplus\epsilon}
\newcommand{\landor}{(\!\land\!)}

\newcommand{\regnum}{k}
\newcommand{\Ij}{I_j}
\newcommand{\qj}{q_j}
\newcommand{\Rj}{R_j}
\newcommand{\Ik}{I_k}
\newcommand{\qk}{q_k^\epsilon}
\newcommand{\Rk}{R_k}
\newcommand{\Jk}{J_k}
\newcommand{\Succ}{S}
\newcommand{\Mk}{M_k}
\newcommand{\Mj}{M_j}
\newcommand{\Inst}{I}
\newcommand{\Is}{I_s}
\newcommand{\States}{Q}
\newcommand{\qop}{{q_0^\epsilon}'}
\newcommand{\qopp}{{q_0^\epsilon}''}
\newcommand{\cl}[1]{\mathsf{c}(#1)}
\newcommand{\oc}[1]{\hat{#1}}
\newcommand{\occ}[1]{\oc{\oc{#1}}}

\newcommand{\NOTE}[2]{{\color{blue}#2}\textsuperscript{{\tiny\bf\color{red} {#1}}}}
\newcommand{\CHAT}[2]{{\small\bf\color{red} {#1: }}{\color{blue}#2}}

\newcommand{\ld}{{\backslash}}
\newcommand{\rd}{{/}}
\newcommand{\ra}{\mathbin{\rightarrow}}
\newcommand{\jn}{\vee}
\newcommand{\mt}{\wedge}
\newcommand{\m}{\mathbf}
\newcommand{\ldd}{{\bbslash}}
\newcommand{\rdd}{{\sslash}}
\newcommand{\omt}{\curlywedge}
\newcommand{\rdo}{\rightthreetimes}
\newcommand{\ldo}{\leftthreetimes}
\newcommand{\1}{\varepsilon}
\newcommand{\T}{\epsilon}
\newcommand{\btl}{\triangleleft}
\newcommand{\btr}{\triangleright}
\newcommand{\ua}{{\uparrow}}
\newcommand{\N}{\mathrel{N}}

\maketitle

\section{Introduction}



The logic of bunched implications (BI) was introduced by O'Hearn and Pym~\cite{OHearnP99:jsl} in 1999 as a logic combining additive (intuitionistic) and multiplicative (substructural) connectives. In particular, the logic contains two implications: an intuitionistic implication and a substructural implication. From a proof-theoretic perspective---its original formulation---BI has an elegant definition as the free combination of the sequent calculi for intuitionistic propositional logic and multiplicative intuitionistic linear logic (MILL). 

BI provides an intuitionistic framework for reasoning about resources~\cite{IshOHe2001}. Indeed, the early attention received by BI, and its Boolean counterpart BBI, can be attributed to its prominent role as the assertion logic of separation logic, an extension of Hoare-style reasoning for imperative programs that manipulate pointers and mutable heaps. 
Separation logic~\cite{Reynolds00:intuitionistic,Reynolds02:lics} enabled reasoning about the use of shared data structures, heap mutation, pointer aliasing and (de)allocation of memory, followed by later applications to concurrent dynamic memory management~\cite{Brookes2007, BrookesOH2016}. In these interpretations, the logical connective $\ast$ can be used to make assertions that hold on disjoint portions of the heap, hence enabling local reasoning about a program’s memory footprint~\cite{PymOHY04:tcs}. This approach dramatically simplifies verification of heap-manipulating programs and scales modular reasoning to large codebases by permitting a local focus rather than on the global heap.

The practical significance of separation logic is evident in modern automated verification tools: Infer~\cite{Calcagno2011, InferWWW},  a static analysis tool that is deployed at Facebook/Meta uses separation logic to analyze millions of lines of code, in order to find null-pointer dereferences, memory leaks, and other bugs. This is a striking example of the utility of formal logic in real-world software reliability and safety.

In the decades following its introduction, the interest in BI has extended well beyond its connections to separation logic. This inspired the study of further logics in the vicinity, such as Generalized BI~\cite{GalatosJ}  (non-commutative version of BI), and its counterpart in the classical setting of Boolean algebras with operators \cite{JonssonTarski1951, JonssonTsinakis93}. The algebraic semantics of GBI-logic is given by residuated Heyting algebras, i.e., bounded residuated lattices with a Heyting implication (hence the underlying lattice is distributive). These algebras are called GBI-algebras, and if the Heyting algebra is a Boolean algebra, 
then they are known as residuated monoid algebras or rm-algebras for short. The corresponding logic is abbreviated BGBI and its commutative version is BBI.

Against this backdrop, the computational status of BI has remained its most prominent open question.
Given the apparent simplicity of its sequent calculus, and with provability for its two constituent systems being in PSPACE, there was an expectation---several papers~\cite{GalmicheMP05:mscs,KaminskiFrancez16,GalatosJ} were even published claiming this result---that BI was decidable. This continued to be the case even when the undecidability of BGBI and BBI, first proven in~\cite{KuruczNSS95:jolli}, became widely known through independent rediscoveries (in the commutative case) by \cite{BrotherstonK10:lics,BrotherstonK14:jacm, Larchey-WendlingG10:lics}. 

In this work, we finally resolve the computational status of BI by showing that many bunched implication logics, including BI and GBI as well as their $\{\land, \lor, \neg, \ld, \rd\}$-reducts, are undecidable. The proof proceeds by a reduction from the Wang Tiling Problem and it may be of standalone interest, as cognate proof methods have recently been useful within both relevant and modal logic~\cite{Knudstorp24,Knudstorp25}.

The paper is organized as follows. Section~\ref{sec:SemSem} includes the preliminaries needed to state and clarify the scope of the undecidability result. Section~\ref{sec:Undecidability} contains the main argument showing that the tiling problem can be interpreted into the class of disjointive distributive residuated lattices. In Section~\ref{sec:erroneous}, we point out the gaps in previous claims of decidability for BI. Finally, Section~\ref{sec:duality} provides details of the connection between the algebraic semantics and frame semantics of disjointive distributive residuated lattices.

Additionally, an alternative proof of undecidability for BI---a reduction from the acceptance problem in And-branching Counter Machines (ACMs)---appears in Appendix~\ref{sec-ACM-reduction}.

\section{Preliminaries and Results}\label{sec:SemSem}
In this section, we set out the preliminaries, state the main result, and explain how it applies to BI in particular. The logic BI can be defined in three equivalent ways: via a proof-theoretic sequent calculus (see Figure~\ref{fig-LBI-calculus}), via BI-algebras (presented below), and via its relational semantics (relevant parts discussed later).


\begin{definition}[Formulas, bunches, and sequents]\label{d: fbs}
    For a denumerable set $P$ of propositional letters, the \textit{formulas} of BI are given by the grammar:
    \begin{align*}
        \varphi\mathrel{::=} p\in P \mid \top \mid \bot \mid 1\mid\varphi\land\varphi \mid\varphi\lor\varphi\mid \varphi\to\varphi \mid \varphi\ast\varphi\mid \varphi\sepimp\varphi.
    \end{align*}
    \textit{Bunches} are, in turn, defined as follows:
    \begin{align*}
        \Gamma\mathrel{::=} \varphi \mid \varnothing_+ \mid \varnothing_\times \mid \Gamma; \Gamma\mid\Gamma,\Gamma.
    \end{align*}
    \textit{Sequents} are pairs $\Gamma\DC \varphi$, where $\Gamma$ is a bunch and $\varphi$ is a formula. We write $\vdash_{\text{BI}}\Gamma\DC \varphi$, or simply $\vdash\Gamma\DC \varphi$, if the sequent $\Gamma\DC \varphi$ is  provable in the bunched calculus of BI (given in  Figure~\ref{fig-LBI-calculus}), and $\nvdash\Gamma\DC \varphi$ if it is not.
\end{definition}



\begin{figure*}[t]
\centering
\footnotesize

\begin{tabular}{@{}ccccc@{}}

\AxiomC{}
\RightLabel{\scriptsize ax}
\UnaryInfC{$\phi \Ra \phi$}
\DisplayProof
&
\AxiomC{$\Gamma \Ra \phi$}
\RightLabel{\scriptsize $\Delta \equiv \Gamma$}
\UnaryInfC{$\Delta \Ra \phi$}
\DisplayProof
&
\AxiomC{$\Gamma(\Delta)\Ra \phi$}
\RightLabel{\scriptsize w}
\UnaryInfC{$\Gamma(\Delta;\Delta')\Ra \phi$}
\DisplayProof
&
\AxiomC{$\Gamma(\Delta;\Delta)\Ra \phi$}
\RightLabel{\scriptsize c}
\UnaryInfC{$\Gamma(\Delta)\Ra \phi$}
\DisplayProof
&
\AxiomC{$\Delta \Ra \phi$}
\AxiomC{$\Gamma(\phi)\Ra \psi$}
\RightLabel{\scriptsize cut}
\BinaryInfC{$\Gamma(\Delta)\Ra \psi$}
\DisplayProof
\\[1.1em]

\AxiomC{}
\RightLabel{\scriptsize $\bot_L$}
\UnaryInfC{$\Gamma(\bot)\Ra \phi$}
\DisplayProof
&
\AxiomC{$\Gamma(\varnothing_\times)\Ra \phi$}
\RightLabel{\scriptsize $1_L$}
\UnaryInfC{$\Gamma(1)\Ra \phi$}
\DisplayProof
&
\AxiomC{}
\RightLabel{\scriptsize $1_R$}
\UnaryInfC{$\varnothing_\times \Ra 1$}
\DisplayProof
&
\AxiomC{$\Gamma(\varnothing_+)\Ra \phi$}
\RightLabel{\scriptsize $\top_L$}
\UnaryInfC{$\Gamma(\top)\Ra \phi$}
\DisplayProof
&
\AxiomC{}
\RightLabel{\scriptsize $\top_R$}
\UnaryInfC{$\varnothing_+ \Ra \top$}
\DisplayProof
\\[1.25em]

\AxiomC{$\Gamma(\phi,\psi)\Ra \chi$}
\RightLabel{\scriptsize $*_L$}
\UnaryInfC{$\Gamma(\phi * \psi)\Ra \chi$}
\DisplayProof
&
\AxiomC{$\Gamma \Ra \phi$}
\AxiomC{$\Delta \Ra \psi$}
\RightLabel{\scriptsize $*_R$}
\BinaryInfC{$\Gamma,\Delta \Ra \phi * \psi$}
\DisplayProof
&
\AxiomC{$\Delta \Ra \phi$}
\AxiomC{$\Gamma(\psi)\Ra \chi$}
\RightLabel{\scriptsize $\mimp_L$}
\BinaryInfC{$\Gamma(\Delta,\;\phi \mimp \psi)\Ra \chi$}
\DisplayProof
&
\AxiomC{$\Gamma,\phi \Ra \psi$}
\RightLabel{\scriptsize $\mimp_R$}
\UnaryInfC{$\Gamma \Ra \phi \mimp \psi$}
\DisplayProof
&
\AxiomC{$\Gamma \Ra \phi_i \ (i=1,2)$}
\RightLabel{\scriptsize $\lor_{R i}$}
\UnaryInfC{$\Gamma \Ra \phi_1 \lor \phi_2$}
\DisplayProof
\\[1.25em]

\AxiomC{$\Gamma(\phi;\psi)\Ra \chi$}
\RightLabel{\scriptsize $\land_L$}
\UnaryInfC{$\Gamma(\phi \land \psi)\Ra \chi$}
\DisplayProof
&
\AxiomC{$\Gamma \Ra \phi$}
\AxiomC{$\Delta \Ra \psi$}
\RightLabel{\scriptsize $\land_R$}
\BinaryInfC{$\Gamma;\Delta \Ra \phi \land \psi$}
\DisplayProof
&
\AxiomC{$\Delta \Ra \phi$}
\AxiomC{$\Gamma(\psi)\Ra \chi$}
\RightLabel{\scriptsize $\to_L$}
\BinaryInfC{$\Gamma(\Delta;\;\phi \to \psi)\Ra \chi$}
\DisplayProof
&
\AxiomC{$\Gamma;\phi \Ra \psi$}
\RightLabel{\scriptsize $\to_R$}
\UnaryInfC{$\Gamma \Ra \phi \to \psi$}
\DisplayProof
&
\AxiomC{$\Gamma(\phi)\Ra \chi$}
\AxiomC{$\Gamma(\psi)\Ra \chi$}
\RightLabel{\scriptsize $\lor_L$}
\BinaryInfC{$\Gamma(\phi \lor \psi)\Ra \chi$}
\DisplayProof
\\

\end{tabular}

\caption{The LBI sequent calculus~\cite{GalmicheMP05:mscs}.
The $\equiv$ denotes commutative monoid equations for $,$ and $;$ applied to sub-bunches.}
\label{fig-LBI-calculus}
\end{figure*}

It is easy to see that a sequent $\Gamma\DC \varphi$ is provable in the BI calculus iff the sequent $\widehat{\Gamma} \DC \varphi$ is provable iff the sequent $\top \DC \widehat{\Gamma} \imp \varphi$ is provable, where $\widehat{\Gamma}$ is the formula obtained from the bunch $\Gamma$ by replacing comma by $\ast$, semicolon by $\mt$, $\varnothing_+$ by $\top$, and $\varnothing_\times$ by $1$. 
 As a result, the decidability of BI is equivalent to deciding sequents of the form $\top \DC \varphi$.
 
With this, we turn to the algebraic semantics of BI.

\begin{definition}[BI-algebras]
    A \textit{BI-algebra} is an algebra of the form $\mathbf{A}=(A, \land, \lor, \to, \top, \bot, \ast, \sepimp, 1)$ where $(A, \land, \lor, \to, \top, \bot)$ is a Heyting algebra, $(A, \ast, 1)$ is a commutative monoid, and $\ast$ is residuated by $\sepimp$; that is, for all $x,y,z\in A$,
    \[
        x\ast y\leq z\qquad \text{iff}\qquad y\leq x \sepimp z,
    \]
    where $\leq$ is the lattice order. A \emph{BBI-algebra} is a BI-algebra where the underlying Heyting algebra is Boolean.

\end{definition}
\begin{remark}[The algebra $\mathcal{P}_{\omega}(\mathbb{N})^+$]\label{rm:fincofin}
    A specific BI-algebra will play a central role in our proof, as its equational theory will provide an upper bound for an interval of undecidable theories that includes BI. To define it, let $\mathcal{P}_{\omega}(\mathbb{N})\mathrel{:=}\{X\subseteq \mathbb{N}\mid \textit{X} \text{ is finite}\}$ be the set of finite sets of natural numbers. Taking its powerset $\mathcal{P}(\mathcal{P}_{\omega}(\mathbb{N}))$, we form the algebra $\mathcal{P}_{\omega}(\mathbb{N})^+$ by equipping it with operations as follows:
    \[
       \mathcal{P}_{\omega}(\mathbb{N})^+\mathrel{:=}\big(\mathcal{P}(\mathcal{P}_{\omega}(\mathbb{N})), \cap, \cup, \to,\mathcal{P}_{\omega}(\mathbb{N}) , \varnothing, \ast, \sepimp, \{\varnothing\}\big )
    \]
    where
    \begin{itemize}
        \item $\to$ is Boolean implication, i.e., $X\to Y\mathrel{:=}X^c\cup Y$
        \item $\ast$ is point-wise union, i.e., $X\ast Y\mathrel{:=}\{x\cup y\mid x\in X, y\in Y\}$
        \item $\sepimp$ is the residual of $\ast$, i.e., $X\sepimp Y\mathrel{:=}\{z\mid \text{for all }x\in X{:}\; z\cup x\in Y\}$.
    \end{itemize}
    That $\mathcal{P}_{\omega}(\mathbb{N})^+$ is a BBI-algebra follows from the fact that it has a Boolean reduct and it is the powerset algebra of a commutative monoid, so residuation holds (cf., e.g., Theorem 3.32 of \cite{GalatosJKO07}).
\end{remark}
It is well known (e.g., see \cite{GalatosJKO07}) that residuation can be captured equationally, so the class $\mathsf{BI}$ of all BI-algebras is a variety by Birkhoff's theorem. Cf. \cite{PymOHY04:tcs}, the variety $\mathsf{BI}$ serves as algebraic semantics for BI, in that for all BI-formulas $\varphi$ and $\psi$, 
\[
    \vdash \varphi\DC \psi\qquad \text{iff} \qquad \mathsf{BI}\vDash \varphi\leq \psi.
\]
Here, by $\mathsf{BI}\vDash \varphi\leq\psi$, we mean that $\m A\vDash \varphi\leq\psi$ for all BI-algebras~$\mathbf{A}$; this, in turn, means that for all homomorphisms $h$ from the BI-formula (term) algebra to $\mathbf{A}$, it holds that $h(\varphi)\leq h(\psi)$.
Combined with previous observations, we have  $\vdash_{\text{BI}}  \Gamma\DC \psi$  iff $\mathsf{BI}\vDash \widehat{\Gamma} \leq \psi$.

Of interest to us will be the special (and, as discussed before, equivalent) case where $\Gamma=\top$. For brevity, we write $\mathbf{A}\nvDash \varphi$ if $\mathbf{A}\nvDash \top\leq \varphi$, and
we say that the BI-algebra $\mathbf{A}$ \textit{refutes} the formula $\varphi$; we say that $\mathsf{BI}$ refutes $\varphi$, written $\mathsf{BI}\nvDash \varphi$, if there is some BI algebra that refutes $\varphi$.
It then follows that
\begin{align}
    \nvdash\top \DC \varphi\qquad \text{iff}\qquad \mathsf{BI}\nvDash \varphi.\footnotemark\label{eq:BIUnd}
\end{align}
We\footnotetext{Note that we do not denote $\top\DC \varphi$ by the shorter $\DC \varphi$, as the latter is ambiguous between $\varnothing_+\DC \varphi$ and $\varnothing_\times\DC \varphi$, corresponding to $\top\DC \varphi$ and $1\DC \varphi$, respectively.} will show that it is undecidable whether, given input $\varphi$, $\mathsf{BI}\nvDash \varphi$---i.e., that $\mathsf{BI}$ has an undecidable equational theory---and thus get the undecidability of provability for the BI-calculus.

However, the scope is broader: our main theorem establishes undecidability for many systems besides BI, both weaker and stronger. Among these, of perhaps particular interest is the non-commutative variant of BI studied in  \cite{GalatosJ, JipsenLitak2022}, which does not assume commutativity of $\ast$, denotes it by $\cdot$ instead, and trades one residual ($\sepimp$) for two: a left division ($\backslash$) and a right division ($/$). The resulting algebras, called GBI-algebras (generalized bunched implication algebras), form a variety $\mathsf{GBI}$ and are defined as follows.
\begin{definition}[GBI-algebras]
    A \textit{GBI-algebra} is an algebra of the form $\mathbf{A}=(A, \land, \lor, \to, \top, \bot, \cdot, \backslash, /, 1)$ where $(A, \land, \lor, \to, \top, \bot)$ is a Heyting algebra, $(A, \cdot, 1)$ is a monoid, and $\backslash, \slash$ are the left and right divisions of $\cdot$; that is, for all $x,y,z\in A$,
    \[
        x\cdot y\leq z\qquad 
        \text{iff}\qquad y\leq x\backslash z\qquad \text{iff}\qquad x\leq z /y.
    \]
\end{definition}
Observe that if $\cdot$ is commutative, then $x\backslash y=y/x$, whence BI-algebras are precisely the commutative GBI-algebras.

    For undecidability, a much weaker setting suffices. In Heyting algebras---the additive part of (G)BI-algebras---the intuitionistic negation is defined by $\neg x\mathrel{:=}x\to \bot$ and satisfies the laws of a pseudocomplement, i.e., $x \mt y \leq \bot$ iff $y \leq \neg x$.  Every Heyting algebra $(A, \land, \lor, \to, \top, \bot)$ is therefore a pseudocomplemented distributive lattice $(A, \land, \lor, \neg, \top, \bot)$, but not conversely: pseudocomplemented distributive lattices $(A, \land, \lor, \neg, \top, \bot)$ need not have a residual ($\to$) to $\land$. Our proof, however, neither requires a Heyting implication ($\to$) nor a full pseudocomplement (nor even the multiplicative unit $1$). It suffices to have a unary operation $\neg$ validating \textit{explosion}, i.e., $x\land \neg x\leq y$. In this case $\bot:=x\land \neg x$ becomes a definable constant;  alternatively and equivalently, we may add a primitive constant $\bot$ in the language and stipulate the explosion equation in the form $ x\land \neg x =\bot$. In any case, the explosion equation is strictly weaker than the pseudocomplementation demand, as for example in the latter case the negation operation is also antitone and satisfies double-negation introduction. As the equation $ x\land \neg x =\bot$ corresponds to the fact that a set is disjoint from its complement, we refer to it as the \emph{disjointive} equation. The weakest algebras of concern are thus the following.
\begin{definition}
    A 
    \textit{disjointive distributive residuated lattice} is an algebra of the form $(A, \land, \lor, \top, \bot, \neg, \cdot, \backslash, /)$  where $(A, \land, \lor, \top, \bot)$ is a bounded distributive lattice, $(A, \cdot)$ is a semigroup, $\backslash, \slash$ residuate $\cdot$, and $\neg$ is 
    a disjointive operation,
    i.e., for all $x\in A$,
    \[
        x\land \neg x =\bot.\qedhere
    \]
\end{definition}
Observe that (G)BI-algebras (or their appropriate reducts, which we will conflate when harmless) are, in particular, disjointive distributive residuated lattices. Specifically, the BI-algebra $\mathcal{P}_{\omega}(\mathbb{N})^+$ from Remark~\ref{rm:fincofin} is a disjointive distributive residuated lattice.

As a final generalization worth mentioning, we obtain undecidability already in the fragment of the language without $\cdot$; algebraically, this corresponds to the equational theories of the $\{\cdot\}$-free reducts. In the case of BI, cf. \eqref{eq:BIUnd}, this means undecidability of deciding whether $\vdash\top\DC \varphi$ (or $\vdash\varnothing_+\DC \varphi$), when the input $\varphi$ is a formula in the language $\{\top, \bot, \land, \lor, \neg, \sepimp\}$ (where $\neg\alpha\mathrel{:=}\alpha\to \bot$).\footnote{\label{fn: language}In fact, we even show undecidability for formulas $\varphi$ in the language without the bounds, $\{\land, \lor, \neg, \sepimp\}$, albeit $\bot$ is definable in this language as $x\land \neg x$ and $\top$ as $\bot\sepimp\bot$.}
\\\\
This explains the general setting to which our undecidability result pertains. The result is obtained by a reduction from the \textit{Wang Tiling Problem}, formulated by \cite{Wang1963}, which we now proceed to define.
\begin{definition}[Wang tiling]
A \textit{(Wang) tile} is a 4-tuple  $$t=(t_1,t_2,t_3,t_4)\in\mathbb{N}^4.$$ 
We think of $t$ as a square tile with `colors' $t_1,t_2,t_3,t_4$ on its up, down, left, and right edge, respectively, and we define $\text{U}(t)=t_1$, $\text{D}(t)=t_2$, $\text{L}(t)=t_3$, $\text{R}(t)=t_4$.

    Given a finite set $\mathcal{W}$ of tiles, we say that a function $\tau:\mathbb{N}^2\to \mathcal{W}$ is a \textit{tiling (of $\mathbb{N}^2$ with $\mathcal{W}$)}, if $\tau$ assigns matching colors to the common sides of adjacent tiles, i.e., if for all $(m,n)\in \mathbb{N}^2$,
    \begin{align*}
        \text{U}(\tau(m,n))=\text{D}(\tau(m,n+1)) \quad  \text{and} \quad \text{R}(\tau(m,n))=\text{L}(\tau(m+1,n)).
    \end{align*}
    We say that a finite set of tiles $\mathcal{W}$ is a \textit{(Wang) tiling} or that \emph{$\mathcal{W}$ tiles $\mathbb{N}^2$}, if there is a tiling function $\tau:\mathbb{N}^2\to \mathcal{W}$ (see Fig.~\ref{fig:tiles}).

    \textit{The (Wang) tiling problem} takes as input a finite set of tiles $\mathcal{W}$ and asks whether $\mathcal{W}$ is a tiling. 
\end{definition}

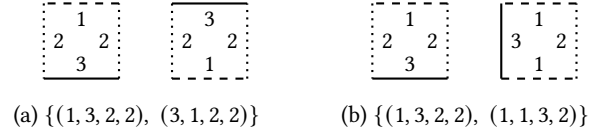
\begin{figure}
\begin{center}
\begin{tikzpicture}
\node at(.5,.8){1};\draw[thick, dashed](1,1)--(0,1);
\node at(.8,.5){2};\draw[thick, dotted](1,0)--(1,1);
\node at(.5,.2){3};\draw[thick](0,0)--(1,0);
\node at(.2,.5){2};\draw[thick, dotted](0,1)--(0,0);
\node at(.5,-.5){(a) $\{(1,3,2,2),$};
\end{tikzpicture}\!\!\!
\begin{tikzpicture}
\node at(.5,.8){3};\draw[thick](1,1)--(0,1);
\node at(.8,.5){2};\draw[thick, dotted](1,0)--(1,1);
\node at(.5,.2){1};\draw[thick, dashed](0,0)--(1,0);
\node at(.2,.5){2};\draw[thick, dotted](0,1)--(0,0);
\node at(.5,-.5){$(3,1,2,2)\}$};
\end{tikzpicture}
\ \qquad\ 
\begin{tikzpicture}
\node at(.5,.8){1};\draw[thick, dashed](1,1)--(0,1);
\node at(.8,.5){2};\draw[thick, dotted](1,0)--(1,1);
\node at(.5,.2){3};\draw[thick](0,0)--(1,0);
\node at(.2,.5){2};\draw[thick, dotted](0,1)--(0,0);
\node at(.5,-.5){(b) $\{(1,3,2,2),$};
\end{tikzpicture}\!\!\!
\begin{tikzpicture}
\node at(.5,.8){1};\draw[thick, dashed](1,1)--(0,1);
\node at(.8,.5){2};\draw[thick, dotted](1,0)--(1,1);
\node at(.5,.2){1};\draw[thick, dashed](0,0)--(1,0);
\node at(.2,.5){3};\draw[thick](0,1)--(0,0);
\node at(.5,-.5){$(1,1,3,2)\}$};
\end{tikzpicture}
\caption{Examples of (a) tiles that tile $\mathbb{N}^2$, and of (b) tiles that do not.}\label{fig:tiles}
\end{center}
\end{figure}

By assigning to every Turing machine $T$ a  set of tiles $\mathcal{W}^T$ and showing that a non-terminating run of $T$ yields a Wang tiling of $\mathbb{N}^2$ with $\mathcal{W}^T$, \cite{Berger} proved the following.

\begin{theorem}[\cite{Berger}]\label{t: tiling problem}
    The tiling problem is undecidable.
\end{theorem}
We achieve undecidability of BI by reducing the tiling problem to BI's provability problem. In brief, we computably associate to each finite set of tiles $\mathcal{W}$ a formula $\phi^{}_{\mathcal{W}}$, and show that $\mathcal{W}$ tiles $\mathbb{N}^2$ iff $\mathsf{BI}\nvDash \phi^{}_{\mathcal{W}}$. Thus, cf. \eqref{eq:BIUnd}, a Turing machine $T$ terminates iff $\vdash \top \DC \phi^{}_{\mathcal{W}^T}$. 

En route, we prove two key lemmas, which serve to give a lower bound and an upper bound, respectively, for an interval of undecidable theories. We state them here to show how they imply undecidability of BI, but postpone their proofs to the subsequent section.

\begin{lemma}\label{l: backward}
    Let $\mathcal{W}$ be a finite set of tiles, and $\mathbf{A}$ a disjointive distributive residuated lattice. If $\mathbf{A}\nvDash \phi^{}_\mathcal{W}$, then $\mathcal{W}$ tiles $\mathbb{N}^2$.
\end{lemma}
\begin{lemma}\label{l: forward}
    Let $\mathcal{W}$ be a finite set of tiles. If $\mathcal{W}$ tiles $\mathbb{N}^2$, then $\mathcal{P}_{\omega}(\mathbb{N})^+\nvDash \phi^{}_\mathcal{W}$.
\end{lemma}
Combined, the lemmas lead to our main theorem.

\begin{theorem}\label{t: main}
    Every class of disjointive distributive residuated lattices that contains $\mathcal{P}_{\omega}(\mathbb{N})^+$ has an undecidable equational theory.
\end{theorem}
\begin{proof}
    Let $\mathcal{K}$ be the class and let $\mathcal{W}$ be a finite set of Wang tiles. By Lemma~\ref{l: backward}, if $\mathcal{K}\nvDash \phi^{}_\mathcal{W}$, then $\mathcal{W}$ tiles $\mathbb{N}^2$. Conversely, by Lemma~\ref{l: forward}, if $\mathcal{W}$ tiles $\mathbb{N}^2$, then $\mathcal{P}_{\omega}(\mathbb{N})^+\nvDash \phi^{}_\mathcal{W}$, so $\mathcal{K}\nvDash \phi^{}_\mathcal{W}$. Thus, $\mathcal{W}$ tiles $\mathbb{N}^2$ iff $\mathcal{K}\nvDash \phi^{}_\mathcal{W}$ (that is, iff $\mathcal{K}\nvDash \top \leq \phi^{}_\mathcal{W}$), whence the tiling problem reduces to the equational decision problem for $\mathcal{K}$.
\end{proof}

\begin{theorem}
    BI is undecidable.
\end{theorem}
\begin{proof}
    $\mathcal{P}_{\omega}(\mathbb{N})^+$ is a BI-algebra and BI-algebras are, in particular, disjointive distributive residuated lattices. Consequently, Theorem~\ref{t: main} applies. Specifically, cf. \eqref{eq:BIUnd}, it is undecidable whether, given a formula $\varphi$, the sequent $\top\DC \varphi$ is derivable in the BI calculus.
\end{proof}
As we will see, the tiling formulas $\phi^{}_\mathcal{W}$ will actually use only the language $\{\land, \lor, \neg, \backslash, \slash\}$, which corresponds to $\{\land, \lor, \neg, \sepimp\}$ in the commutative setting. As a result, undecidability holds already in this fragment, as we discussed in the paragraph leading to footnote~\ref{fn: language}.

Further, as GBI-algebras are disjointive distributive residuated lattices, and $\mathcal{P}_{\omega}(\mathbb{N})^+$---as a BI-algebra---is a GBI-algebra, Theorem~\ref{t: main} also entails undecidability of the equational theory of $\mathsf{GBI}$.
\begin{theorem}
    GBI is undecidable.
\end{theorem}
Similarly, since  $\mathcal{P}_{\omega}(\mathbb{N})^+$ is a BBI-algebra, we attain a proof that BBI (Boolean Bunched Implication Logic) is undecidable, a result earlier achieved in \cite{KuruczNSS95:jolli, BrotherstonK10:lics, Larchey-WendlingG10:lics, Knudstorp25}. 
\begin{theorem}
    BBI is undecidable.
\end{theorem}
Additionally, arguments of \cite{Knudstorp25} readily transfer to our context, leading to undecidability of many variants of BBI, notably within the $\{\ast\}$-free fragment. 

Lastly, there is one final consequence of Theorem~\ref{t: main} we wish to highlight. \cite{Kozak09} showed that the variety of distributive residuated lattices (which include a multiplicative unit $1$) has a \textit{decidable} equational theory (hence also the class of subreducts without unit). In contrast, from Theorem~\ref{t: main}, it follows that the variety of disjointive distributive residuated lattices (with or without a multiplicative unit $1$) has an \textit{undecidable} equational theory. 
\begin{theorem}
    The variety of disjointive distributive residuated lattices has an undecidable equational theory.
\end{theorem}
That is, the mere presence of an operation $\neg$ with $x\land \neg x=\bot$ is enough to cross from the decidable to the undecidable.

\section{Tiling proofs}\label{sec:Undecidability}
We are left to prove Lemma~\ref{l: backward} and~\ref{l: forward}. To this end, we work with dual, relational structures, which we call disjointive associative frames and define as follows. For readers familiar with relational semantics for BI, we mention that these generalize the upwards and downwards closed monoidal frames (UDMF), which form a complete semantics for BI \cite{DochPym2019}.

\begin{definition}[Frames and models]
    A \textit{(disjointive associative) frame} is a triple $\mathfrak{F}=(S,\circ, N)$ where $\circ: S^2\to \mathcal P(S)$ is associative and $N:\mathcal P(S)\to \mathcal  P(S)$ is \emph{disjointive}, i.e., for all $x,y,z\in S$ and $X\subseteq S,$
\[
    (x\circ y)\circ z=x\circ(y\circ z)\qquad \text{and}\qquad X\cap N(X)=\varnothing.
\]
Here, $X\circ Z=\bigcup\{x\circ z\mid x\in X, z\in Z\}$, $X\circ z=X\circ\{z\}$ and $x\circ Z=\{x\}\circ Z$ for $X,Z\subseteq S$. Note that the notation $z\in x\circ y$  is equivalent to (and more convenient than) the more traditional ternary relation notation $R(x,y,z)$.
    
A \textit{(disjointive associative) model} is a pair $\mathfrak{M}=(\mathfrak{F}, V)$ where $\mathfrak{F}=(S,\circ, N)$ is a frame and $V$ is a valuation on $S$, i.e., a function $V:P\to \mathcal{P}(S)$.
\end{definition}

\begin{definition}[Satisfaction and refutation]\label{def-satisfaction}
    For models $\mathfrak{M}=(S, \circ,N, V)$ and formulas $\varphi$ of the language $\{\land, \lor, \neg, \backslash, /\}$, we define the \textit{satisfaction set} of $\varphi$ (w.r.t. $\mathfrak{M}$), written $\|\varphi\|_\mathfrak{M}$ or just $\|\varphi\|$, recursively as follows.
    \begin{alignat*}{3}
        \| p\|&\mathrel{:=}V(p) &\qquad 
        \|\neg \varphi\|&\mathrel{:=}N(\|\varphi\|) 
        \\
        \| \varphi \land \psi \|&\mathrel{:=}\|\varphi\|\cap \|\psi\| &\qquad
        \| \varphi \backslash \psi \|&\mathrel{:=}\{s\in S\mid \|\varphi\|\circ s\subseteq \|\psi\|\} \\
        \| \varphi\lor \psi \|&\mathrel{:=}\|\varphi\|\cup \|\psi\|&\qquad 
        \| \psi \slash \varphi \|&\mathrel{:=}\{s\in S\mid s\circ \|\varphi\|\subseteq \|\psi\|\}.
    \end{alignat*}
    This corresponds to the following point-wise definition of satisfaction $s\in \|\varphi\|$, written $\mathfrak{M}, s\Vdash \varphi$ or simply $s\Vdash\varphi$.
    \begin{align*}
        &
        s\Vdash p && \text{iff} && s\in V(p)\\
        &
        s\Vdash \neg\varphi && \text{iff} && s\in N(\|\varphi\|)\\
        &
        s\Vdash \varphi\land\psi && \text{iff} && 
        s\Vdash \varphi \text{\hspace{0.1cm} and \hspace{0.1cm}} 
        s\Vdash \psi\\
        &
        s\Vdash \varphi\lor\psi && \text{iff} && 
        s\Vdash \varphi \text{\hspace{0.1cm} or \hspace{0.1cm}} 
        s\Vdash \psi\\
        &
        s\Vdash \varphi\backslash\psi && \text{iff} && \text{$\forall x,z\in S$: if } 
        x\Vdash \varphi\text{ and $z\in x\circ s$, then } 
        z\Vdash \psi\\
        &
        s\Vdash \psi\slash\varphi && \text{iff} && \text{$\forall y,z\in S$: if } 
        y\Vdash \varphi\text{ and $z\in s\circ y$, then } 
        z\Vdash \psi.
    \end{align*}
    A model $\mathfrak{M}$ is said to \textit{refute} a formula $\varphi$ or that $\varphi$ \emph{fails} in $\mathfrak{M}$, written $\mathfrak{M}\nvDash\varphi$, if $\|\varphi\|_\mathfrak{M}\neq S$, i.e., if there is $s\in S$ such that $\mathfrak{M},s\nVdash \varphi$. We say that a frame $\mathfrak{F}$  \textit{refutes} a formula $\varphi$, or that $\varphi$ \emph{fails} in $\mathfrak{F}$, and we write $\mathfrak{F}\nvDash\varphi$, if there is a valuation $V$ such that  $(\mathfrak{F}, V) \nvDash\varphi$.
\end{definition}
Take note that the definition of negation in terms of a disjointive operation precisely ensures that $\mathfrak{M}, s\Vdash \neg\varphi$ implies $\mathfrak{M}, s \not \Vdash \varphi$.

\begin{remark}\label{r: fincofindual}
Observe that if $\mathfrak{F}=(S,\circ, N)$ is a disjointive associative frame, then
$$\mathfrak{F}^+\mathrel{:=}\big(\mathcal{P}(S), \cap, \cup,S, \varnothing, \neg, \cdot, \ld, \rd \big )$$ is a disjointive distributive residuated lattice, where $X\cdot Y:= X \circ Y$, $X \ld Y:= \{z \mid X \circ z \subseteq Y\}$, $Y\rd X:= \{z \mid z \circ X \subseteq Y\}$, and $\neg X:=N(X)$ (see, e.g., Section~3.4.10 of \cite{GalatosJKO07}) and that the definition of satisfaction set directly reflects these operations; we call this the \emph{complex algebra} of $\mathfrak{F}$. Also, note that valuations $V:P\to \mathcal{P}(S)$ on the frame are in bijective correspondence with homomorphisms $h$ from the formula algebra to the complex algebra (given $V$, define $h$ by $p\mapsto V(p)=\|p\|$; and, vice versa, given $h$ define $V$ by $p\mapsto h(p)$), and we have $h(\varphi)=\|\varphi\|$ for all formulas $\varphi$ in the language $\{\land, \lor, \neg, \backslash, \slash\}$. It follows that $\mathfrak{F}^+\nvDash \varphi$ iff $\mathfrak{F}\nvDash \varphi$; i.e., the pointwise satisfaction relation $\Vdash$ simply reflects the algebraic satisfaction of the complex algebra.

In particular, the algebra $\mathcal{P}_{\omega}(\mathbb{N})^+$ of Remark~\ref{rm:fincofin} arises as the complex algebra of the disjointive associative frame $(\mathcal{P}_{\omega}(\mathbb{N}), \cup, {}^c)$, where $\circ$ is defined as $\cup$ and $N$ is the complementation operation $^c$, so $u\in s\circ t \iff u=s\cup t$, and $N(X)=X^c$. Hence, to prove Lemma~\ref{l: forward}, it is enough to show that if $\mathcal{W}$ tiles $\mathbb{N}^2$ then $(\mathcal{P}_{\omega}(\mathbb{N}), \cup, {}^c)\nvDash \phi^{}_\mathcal{W}$. We show this in Lemma~\ref{l: forward2} (which thereby is an equivalent, relational formulation of Lemma~\ref{l: forward}).\footnote{For those familiar with relational semantics for BI, it may be of interest to observe that every model $(\mathcal{P}_{\omega}(\mathbb{N}), \cup, {}^c, V)$ defines an upwards and downwards closed monoidal model $(\mathcal{P}_{\omega}(\mathbb{N}),=, \cup, \{\varnothing\}, V)$ for BI (and BBI). In fact, this can be used to provide an alternative proof of the undecidability of BI that circumvents algebraic semantics by directly employing relational semantics. Because we then have $(\mathcal{P}_{\omega}(\mathbb{N}), \cup, {}^c)\nvDash \phi^{}_\mathcal{W}$ implies $(\mathcal{P}_{\omega}(\mathbb{N}),=, \cup, \{\varnothing\})\nvDash \phi^{}_\mathcal{W}$ implies $\nvdash \top\DC \phi^{}_\mathcal{W}$. Together with Lemma~\ref{l: forward2}, this shows that if $\mathcal{W}$ tiles $\mathbb{N}^2$ then $\nvdash \top\DC \phi^{}_\mathcal{W}$. The converse direction (if $\nvdash \top\DC \phi^{}_\mathcal{W}$, then $\mathcal{W}$ tiles $\mathbb{N}^2$) can likewise be proven by employing UDMF semantics and without appeal to algebraic semantics, as we elaborate on below.}
\end{remark}
Conversely, it is possible to define a disjointive associative frame from every disjointive distributive residuated lattice and establish a completeness result. We do this in the next lemma and defer its proof to Section~\ref{sec:duality}.

\begin{lemma}\label{l: dual}
    If a formula in the language $\{\land, \lor, \neg, \backslash, \slash\}$ is refuted by a disjointive distributive residuated lattice, then it fails in a disjointive associative frame.
\end{lemma}

For the benefit of readers who are familiar with the UDMF semantics of BI, we mention that for the purposes of proving undecidability for BI the following specialization of Lemma~\ref{l: dual} is enough; we include a proof of it also in Section~\ref{sec:duality}.

\begin{lemma}\label{l: dualBI}
    If a formula in the language $\{\land, \lor, \neg, \mimp\}$ is refuted by a BI-algebra, then it fails in a disjointive associative frame. 
\end{lemma}

To prove Lemma~\ref{l: backward}, it therefore suffices to show that for every set of tiles $\mathcal{W}$, if $\phi^{}_\mathcal{W}$ fails in a disjointive associative frame, then $\mathcal{W}$ tiles $\mathbb{N}^2$. This is precisely the content of our later Lemma~\ref{l: backward2}. 
\\\\
We continue with the definition of $\phi^{}_{\mathcal{W}}$ given a finite set of tiles $\mathcal{W}$. We use the convention that $\neg$ has highest binding power, followed by $\backslash, /$ and then by $\wedge,\vee$.

\begin{definition}[Tiling formulas]   
    Given a  finite set $\mathcal{W}$ of tiles, let $\{1, \hdots, k\}$ be the finite set of colors involved, where $k \in \mathbb{Z}^+$, and we introduce propositional letters $u_1, \hdots, u_k$, $d_1, \hdots, d_k$, $l_1, \hdots, l_k$, $r_1, \hdots, r_k$. For a tile $t=(t_1,t_2,t_3,t_4)$, by abusing notation, we also write $t$ for the following corresponding conjunction of literals:
    $$
    \bigg(u_{t_1}\land \bigwedge_{j\neq t_1}\neg u_{j}\bigg)\land \bigg(d_{t_2}\land \bigwedge_{j\neq t_2}\neg d_{j}\bigg)\land \bigg(l_{t_3}\land \bigwedge_{j\neq t_3}\neg l_{j}\bigg)\land \bigg(r_{t_4}\land \bigwedge_{j\neq t_4}\neg r_{j}\bigg).$$
    
    Furthermore, we will include the following propositional letters:
    \begin{itemize}
        \item $\textsc{x}$, used to encode elements, composition by which can increment the $x$-coordinate, from $(m,n)$ to $(m+1,n)\in \mathbb{N}^2$. 
        \item $\textsc{y}$, used similarly to encode elements that increment along the $y$-axis.
        \item $\textsc{c}$, used as a gadget to relate arbitrarily long sequences of elements in the model to the initial element via associativity.
        \item $\textsc{ee}, \textsc{oe}, \textsc{oo}$ and $\textsc{eo}$, to be intuited as `a combination of an even number of $\textsc{x}$s and an even number of $\textsc{y}$s', $\hdots$,  `a combination of an even number of $\textsc{x}$s and an odd number of $\textsc{y}$s', respectively.
    \end{itemize}
    Using these propositional letters, we abbreviate: 
    \begin{itemize}
        \item $\textsc{x}'\mathrel{:=}\textsc{x}\land \textsc{c}\land \textsc{c}\backslash \textsc{c}$ 
        \item $\textsc{y}'\mathrel{:=}\textsc{y}\land \textsc{c}\land \textsc{c}\backslash \textsc{c}$ 
        \item $ \textsc{e}\textsc{e}^c\mathrel{:=} \textsc{o}\textsc{e} \lor \textsc{o}\textsc{o} \lor \textsc{e}\textsc{o}$ 
        \item $ \textsc{o}\textsc{e}^c\mathrel{:=} \textsc{e}\textsc{e} \lor \textsc{o}\textsc{o} \lor \textsc{e}\textsc{o}$
        \item $ \textsc{o}\textsc{o}^c\mathrel{:=} \textsc{e}\textsc{e} \lor \textsc{o}\textsc{e} \lor \textsc{e}\textsc{o}$
        \item $ \textsc{e}\textsc{o}^c\mathrel{:=} \textsc{e}\textsc{e} \lor \textsc{o}\textsc{e} \lor \textsc{o}\textsc{o}$
    \end{itemize}
    Here, the added superscripts $^c$ are shorthand for a notational `complement'.
We further abbreviate:
{\small
    \begin{align*}
        &\alpha_{\textsc{ee}}:= \ 
        \textsc{ee}
        \land \neg\textsc{oe}\land\neg\textsc{oo}\land\neg \textsc{eo}\land \neg(\textsc{x}'\backslash \textsc{oe}^c)\,\land\bigvee_{t\in \mathcal{W}}\bigg(t\land \textsc{x}\backslash \Big[\textsc{e}\textsc{e}\lor \big(\textsc{o}\textsc{e}\land \bigvee_{t'\in\mathcal{W}}^{\text{R}(t)=\text{L}(t')}t'\big)\Big]
        \land \Big[\textsc{e}\textsc{e}\lor \big(\textsc{e}\textsc{o}\land \bigvee_{t''\in\mathcal{W}}^{\text{U}(t)=\text{D}(t'')}t''\big)\Big]/\textsc{y}\bigg)\\[8pt]
        &\alpha_{\textsc{oe}}:= \ 
        \textsc{oe} \land\neg\textsc{oo}\land\neg\textsc{eo}\land\neg\textsc{ee}\land\neg(\textsc{oo}^c/\textsc{y}')\,\land\bigvee_{t\in \mathcal{W}}\bigg(t\land\textsc{x}\backslash \Big[\textsc{o}\textsc{e}\lor \big(\textsc{e}\textsc{e}\land \bigvee_{t'\in\mathcal{W}}^{\text{R}(t)=\text{L}(t')}t'\big)\Big]
        \land \Big[\textsc{o}\textsc{e}\lor \big(\textsc{o}\textsc{o}\land \bigvee_{t''\in\mathcal{W}}^{\text{U}(t)=\text{D}(t'')}t''\big)\Big]/\textsc{y} \bigg)\\[8pt]
        &\alpha_{\textsc{oo}}:= \ 
        \textsc{oo}\land\neg\textsc{eo}\land\neg\textsc{ee}\land\neg \textsc{oe}\land \neg(\textsc{x}'\backslash \textsc{eo}^c)\,\land\bigvee_{t\in \mathcal{W}}\bigg(t\land \textsc{x}\backslash \Big[\textsc{oo}\lor \big(\textsc{eo}\land \bigvee_{t'\in\mathcal{W}}^{\text{R}(t)=\text{L}(t')}t'\big)\Big]\land \Big[\textsc{oo}\lor \big(\textsc{oe}\land \bigvee_{t''\in\mathcal{W}}^{\text{U}(t)=\text{D}(t'')}t''\big)\Big]/\textsc{y}\bigg)\\[8pt]
        &\alpha_{\textsc{eo}}:= \ 
        \textsc{eo} \land\neg\textsc{ee}\land\neg \textsc{oe}\land\neg \textsc{oo}\land \neg(\textsc{ee}^c/\textsc{y}')\,\land\bigvee_{t\in \mathcal{W}}\bigg(t\land \textsc{x}\backslash \Big[\textsc{eo}\lor \big(\textsc{oo}\land \bigvee_{t'\in\mathcal{W}}^{\text{R}(t)=\text{L}(t')}t'\big)\Big]\land \Big[\textsc{eo}\lor \big(\textsc{ee}\land \bigvee_{t''\in\mathcal{W}}^{\text{U}(t)=\text{D}(t'')}t''\big)\Big]/\textsc{y}\bigg)
    \end{align*}
    }
Intuitively, for special `coordinate' points $\langle m,n\rangle$ in a model $\mathfrak{M}$,  $\langle m,n\rangle\Vdash \alpha_{\textsc{EE}}$ states that (a) $m$ and $n$ correspond, respectively, to an even $x$-coordinate and an even $y$-coordinate [i.e, $\textsc{ee}, \neg \textsc{oe},\neg\textsc{oo},\neg \textsc{eo}$], (b) the $x$-coordinate can be incremented to reach (odd, even) [i.e., $\neg(\textsc{x}'\backslash \textsc{oe}^c)$],
and (c) there is some tile $t\in\mathcal{W}$ placed here---by the way a tile is defined, exactly one tile is placed here---such that (i)~an increment in the $x$-coordinate implies that we reach $(\text{odd},\text{even})$, where a unique tile~$t'$ is placed whose left edge matches the right edge of $t$, and (ii)~an increment in the $y$-coordinate implies that we reach $(\text{even},\text{odd})$, where a unique tile~$t''$ holds whose down edge matches the top edge of $t$.

Finally, for any propositional variable~$p$, we let $\alpha\mathrel{:=}\alpha_{\textsc{ee}}\lor\alpha_{\textsc{oe}}\lor\alpha_{\textsc{oo}} \lor \alpha_{\textsc{eo}}$ and 
\begin{equation}\label{eq-phi}
        \phi^{}_\mathcal{W} \mathrel{:=}  \Big[\neg (\textsc{x}'\backslash\textsc{o}\textsc{e}^c) \land \textsc{c}\backslash \alpha\land \textsc{c}\backslash \alpha/\textsc{c}\land 
        \alpha/\textsc{c}\Big]\backslash p.
\end{equation}        
\end{definition}

\begin{lemma}\label{l: backward2}
Let $\mathcal{W}$ be a finite set of tiles and $\mathfrak{F}$ a disjointive associative frame. If $\mathfrak{F}\nvDash\phi^{}_\mathcal{W}$, then $\mathcal{W}$ tiles $\mathbb{N}^2$.
\end{lemma}
\begin{proof}
    Since $\mathfrak{F}\nvDash\phi^{}_\mathcal{W}$, there is a valuation $V$, inducing a model $\mathfrak{M}\mathrel{:=}(\mathfrak{F}, V)$, and a point $s'$ such that 
        $$\mathfrak{M}, s'\nVdash \phi^{}_{\mathcal{W}}.$$
    By the semantics of $\backslash$, this means that, in particular, there is some point $s$ such that
        $$\mathfrak{M}, s\Vdash \neg (\textsc{x}'\backslash\textsc{o}\textsc{e}^c) \land \textsc{c}\backslash \alpha\land \textsc{c}\backslash \alpha/\textsc{c}\land 
        \alpha/\textsc{c}.$$
    From this, we will show that $\mathcal{W}$ tiles $\mathbb{N}^2$. Our proof will proceed as follows: first, we identify elements $x_1, x_2, \hdots, y_1, y_2, \hdots$ within $\mathfrak{M}$ and use them to further identify what we will call `staircase' elements of $\mathfrak{M}$. We denote these staircase elements of $\mathfrak{M}$ by 
    $$\langle1,0\rangle, \langle1,1\rangle, \langle2,1\rangle, \hdots, \langle k,k\rangle, \langle k+1, k\rangle, \hdots,$$ and call them such because they will satisfy the following (see Fig.~\ref{fig:stair}):
    \begin{enumerate}
        \item[(stair)] 
        $\langle 1,0\rangle\in x_1\circ s$,\\
        $\langle 1,1\rangle\in\langle 1,0\rangle\circ y_1$,\\
        $\langle 2,1\rangle\in x_2\circ\langle 1,1\rangle$,\\
        $\hdots$\\
        $\langle k+1,k\rangle\in x_{k+1}\circ\langle k,k\rangle$,\\ 
        $\langle k+1,k+1\rangle\in\langle k+1, k\rangle\circ y_{k+1}, \hdots$
    \end{enumerate}
    \begin{figure}[h]
        \begin{center}
        \begin{tikzpicture}[scale=1.5]
            \node (s) at (0,0){$s$};
            \node (10) at (1,0){$\langle1,0\rangle$}edge["$x_1$"](s);
            \node (11) at (1,1){$\langle1,1\rangle$}edge["$y_1$"](10);
            \node (21) at (2,1){$\langle2,1\rangle$}edge["$x_2$"](11);
            \node (22) at (2,2){$\langle2,2\rangle$}edge["$y_2$"](21);
            \node (32) at (3,2){$\langle3,2\rangle$}edge["$x_3$"](22);
            \node at (2.9,2.5){$\cdot$};
            \node at (3.0,2.6){$\cdot$};
            \node at (3.1,2.7){$\cdot$};
        \end{tikzpicture}\\[10pt]
        \begin{tikzpicture}[scale=1]
        \node (p) at (0,2){$p$};
        \node at (1,2){$q$}edge["$x_i$"](p);
        \node at (1,2)[label=right: means $q\in x_i\circ p$]{};
        \node (q) at (0.5,0){$q$};
        \node at (0.5,1){$r$}edge["$y_i$"](q);
        \node at (1,0.5)[label=right: means $r\in q\circ y_i$]{};
        \end{tikzpicture}
        \caption{Staircase of grid points.}
        \label{fig:stair}
        \end{center}
    \end{figure}
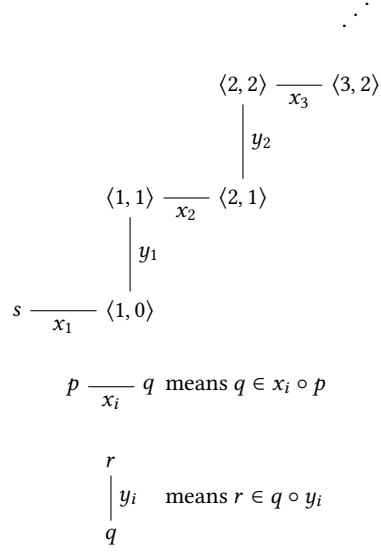
    Next, from the staircase and associativity, we find a full grid of elements $\langle m,n\rangle\in \mathfrak{M}$, for all $m,n \in \mathbb{N}$.\footnote{Actually, we will not define a point $\langle 0,0\rangle$. We could have included it, either by making $\phi^{}_{\mathcal{W}}$ more complicated, or by changing our construction/naming convention. Anyhow, tiling all but $(0,0)$ is obviously equivalent to tiling the full quadrant.} Finally, we associate points in the plane $(m,n)\in \mathbb{N}^2$ with the corresponding grid points $\langle m,n\rangle\in \mathfrak{M}$, and have the tiling of $\mathbb{N}^2$ be determined by what tile formula $t\in \mathcal{W}$ is satisfied at $\langle m,n\rangle\in \mathfrak{M}$ (i.e. $\mathfrak{M}, \langle m,n\rangle \Vdash t$).
    \\\\
    We begin by proving the `(stair)'-claim above. We will do this by induction on the sequence
    \begin{enumerate}    
        \item[(stair)] 
        $\langle 1,0\rangle\in x_1\circ s$\\
        $\langle 1,1\rangle\in\langle 1,0\rangle\circ y_1$\\
        $\langle 2,1\rangle\in x_2\circ\langle 1,1\rangle$\\
        $\hdots$\\
        $\langle k+1,k\rangle\in x_{k+1}\circ\langle k,k\rangle$\\ 
        $\langle k+1,k+1\rangle\in\langle k+1, k\rangle\circ y_{k+1}$\\
        $\hdots$
    \end{enumerate}    
    while simultaneously showing the following for all points of the sequence:
    \begin{itemize}
        \item[]
        $x_{i}\Vdash \textsc{x}'$, $y_{i}\Vdash \textsc{y}'$\\
        $\langle 2i,2i\rangle\Vdash \alpha_{\textsc{ee}}$\\
        $\langle 2i+1,2i\rangle\Vdash \alpha_{\textsc{oe}}$\\
        $\langle 2i+1,2i+1\rangle\Vdash \alpha_{\textsc{oo}}$\\
        $\langle 2i,2i+1\rangle\Vdash \alpha_{\textsc{eo}}$.
    \end{itemize}

For the induction base, since $s\Vdash \neg (\textsc{x}'\backslash\textsc{o}\textsc{e}^c)$, we have that 
        $$s\nVdash \textsc{x}'\backslash\textsc{o}\textsc{e}^c.$$
    This must be witnessed by some points, which we denote $x_1$ and $\langle 1,0\rangle$, i.e., 
        $$\text{$\langle 1,0\rangle\in x_1\circ s$ and $x_1\Vdash \textsc{x}'$, but $\langle 1,0\rangle\nVdash \textsc{o}\textsc{e}^c$.}$$
    Since $x_1\Vdash \textsc{x}'$,  in particular we have that $x_1\Vdash \textsc{c}$. Combined with $s\Vdash \textsc{c}\backslash\alpha$ and $\langle 1,0\rangle\in x_1\circ s$, this implies that
        $$\langle 1,0\rangle\Vdash \alpha_{\textsc{ee}}\lor\alpha_{\textsc{oe}}\lor\alpha_{\textsc{oo}} \lor \alpha_{\textsc{eo}}.$$
    But we established $\langle 1,0\rangle\nVdash \textsc{o}\textsc{e}^c$, hence $\langle 1,0\rangle\nVdash \alpha_{\textsc{ee}}\lor \alpha_{\textsc{oo}}\lor\alpha_{\textsc{eo}}$, so
        $$\langle 1,0\rangle\Vdash \alpha_{\textsc{oe}},$$
    which completes the proof of the induction base. 

The induction step divides into four cases, namely whether we assume the induction hypothesis up to 
\begin{enumerate}
\item[(i)] $\langle 2k+1,2k\rangle \in x_{2k+1}\circ\_$
\item[(ii)] $\langle 2k+1,2k+1\rangle\in\_\circ y_{2k+1}$
\item[(iii)] $\langle 2k,2k-1\rangle\in x_{2k}\circ\_$ \ or 
\item[(iv)] $\langle 2k,2k\rangle\in\_\circ y_{2k}$. 
\end{enumerate}

We prove the first, as the others are analogous. 
So, we assume that for some $k\geq 0$, the induction hypothesis holds for the sequence 
\begin{itemize}\item[]
$\langle 1,0\rangle\in x_1\circ s$\\
$\hdots$\\
$\langle 2k+1,2k\rangle\in x_{2k+1}\circ\langle 2k, 2k\rangle$.
\end{itemize}
(if $k=0$, this is just assuming the induction base). By the induction hypothesis, we  have  $\langle 2k+1,2k\rangle\Vdash \alpha_{\textsc{oe}},$ so in particular
    $$\langle 2k+1,2k\rangle\Vdash \neg(\textsc{o}\textsc{o}^c/\textsc{y}').$$
Thus, $\langle 2k+1,2k\rangle\nVdash \textsc{o}\textsc{o}^c/\textsc{y}'$. This must be witnessed by some points, denoted $y_{2k+1}$ and $\langle 2k+1, 2k+1 \rangle$, so
$$\text{$\langle 2k+1,2k+1\rangle\in\langle 2k+1, 2k \rangle\circ y_{2k+1}$ and $y_{2k+1}\Vdash \textsc{y}'$,}$$
 but $\langle 2k+1,2k+1\rangle\nVdash \textsc{o}\textsc{o}^c$. From the induction hypothesis that
\begin{itemize}\item[]
$\langle 1,0\rangle\in x_1\circ s$\\
$\langle 1,1\rangle\in\langle 1,0\rangle\circ  y_1$\\
$\langle 2,1\rangle\in x_2\circ \langle 1,1\rangle$\\
$\hdots$\\
$\langle 2k+1,2k\rangle\in x_{2k+1}\circ\langle 2k, 2k\rangle$
\end{itemize}
and the just established 
    $$\langle 2k+1,2k+1\rangle\in\langle 2k+1, 2k \rangle\circ y_{2k+1},$$
we derive, by repeated associativity, that 
    $$\langle 2k+1,2k+1\rangle\in x_{2k+1}\circ\dots\circ x_2\circ x_1\circ s\circ y_1\circ y_2\circ \dots\circ  y_{2k+1}.$$
So, 
there exist points $x$ and $y$ such that
    $$\langle 2k+1,2k+1\rangle\in x\circ s\circ y$$
and
    $$x\in x_{2k+1}\circ\dots\circ x_2\circ x_1, \quad y\in y_1\circ y_2\circ \dots\circ  y_{2k+1}.$$
By another induction, using that both $x_i\Vdash \textsc{c}\land \textsc{c}\backslash \textsc{c}$ 
and $y_i\Vdash \textsc{c}\land \textsc{c}\backslash \textsc{c}$
, 
it follows that also
    $$x\Vdash \textsc{c} \text{ and }y\Vdash \textsc{c}.$$
Consequently, as in the induction base, from $s\Vdash \textsc{c}\backslash \alpha/\textsc{c}$ and $\langle 2k+1,2k+1\rangle\in x\circ s\circ y$, we deduce
        $$\langle 2k+1,2k+1\rangle\Vdash \alpha.$$
Because $\langle 2k+1,2k+1\rangle\nVdash \textsc{o}\textsc{o}^c$, we have $\langle 2k+1,2k+1\rangle\nVdash \alpha_{\textsc{ee}}\lor \alpha_{\textsc{oe}}\lor\alpha_{\textsc{eo}}$, and hence
        $$\langle 2k+1,2k+1\rangle\Vdash \alpha_{\textsc{oo}}.$$
This completes the induction step, and establishes the necessary properties of the staircase. (Note that the distinctness of the points $\langle m,n\rangle$ is not asserted, nor is it required in what follows.)

Next, we construct the full grid, starting from the staircase elements and moving in two directions: top-left (above the staircase) and bottom-right (below the staircase); see Figure~\ref{fig:gridConstr}.

    For $m \leq n$, if $\langle m+1, n\rangle\in x_{m+1}\circ\langle m, n\rangle$ and $\langle m+1,n+1 \rangle\in\langle m+1, n\rangle\circ  y_{n+1}$, then $\langle m+1,n+1 \rangle\in (x_{m+1}\circ\langle m, n\rangle)\circ  y_{n+1}$, so by associativity $\langle m+1,n+1 \rangle\in x_{m+1}\circ(\langle m, n\rangle\circ  y_{n+1})$; i.e., there is a point, which we denote $\langle m, n+1\rangle$, such that $\langle m, n+1\rangle\in\langle m, n\rangle\circ  y_{n+1}$ and $\langle m+1,n+1\rangle\in x_{m+1}\circ\langle m, n+1\rangle.$ 
    
    Likewise, for  $m > n$, if $\langle m, n+1\rangle\in \langle m, n\rangle\circ y_{n+1}$ and  $\langle m+1,n+1 \rangle\in x_{m+1}\circ \langle m, n+1\rangle$, 
    then there is a point, denoted $\langle m+1, n\rangle$, such that
    $\langle m+1, n\rangle\in x_{m+1}\circ \langle m, n\rangle$ and $\langle m+1,n+1\rangle\in \langle m+1, n\rangle\circ y_{n+1}.$ 
     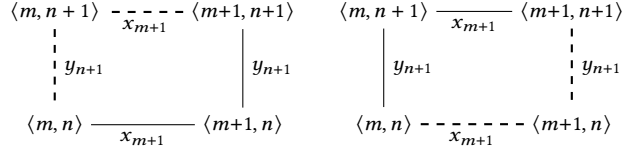
\begin{figure}
        \begin{center}
        \begin{tikzpicture}[xscale=2.5,yscale=1.5]\small
            \node (00) at (0,0){$\langle m, n\rangle$};
            \node (10) at (1,0){$\langle m{+}1,n\rangle$}edge["$x_{m+1}$"](00);
            \node (01) at (0,1){$\langle m,n+1\rangle$}edge[thick, dashed,"$y_{n+1}$"](00);
            \node (11) at (1,1){$\langle m{+}1,n{+}1\rangle$}edge["$y_{n+1}$"](10)edge[thick, dashed,"$x_{m+1}$"](01);
        \end{tikzpicture}
      \quad
        \begin{tikzpicture}[xscale=2.5,yscale=1.5]\small
            \node (00) at (0,0){$\langle m, n\rangle$};
            \node (10) at (1,0){$\langle m{+}1,n\rangle$}edge[thick, dashed,"$x_{m+1}$"](00);
            \node (01) at (0,1){$\langle m,n+1\rangle$}edge["$y_{n+1}$"](00);
            \node (11) at (1,1){$\langle m{+}1,n{+}1\rangle$}edge[thick, dashed,"$y_{n+1}$"](10)edge["$x_{m+1}$"](01);
        \end{tikzpicture}
        \caption{New points: for $m \leq n$ (left) and $m>n$ (right).}
        \label{fig:gridConstr}
        \end{center}
    \end{figure}

Note that for all points of the grid we have:
\begin{center}
    $\langle m+1,n\rangle\in x_{m+1}\circ\langle m, n\rangle$ and $\langle m,n+1\rangle\in\langle m, n\rangle\circ  y_{n+1}$.
\end{center}
Indeed, this holds within the staircase by its  construction and it is ensured across the whole grid by the recursive construction of the newly added elements.

Moreover, using repeated associativity once more, by an implicit induction, we have that for all grid points $\langle m, n\rangle$, 
$$\langle m,n\rangle\in x_{m}\circ\dots\circ x_2\circ x_1\circ s\circ y_1\circ y_2\circ \dots\circ  y_{n},$$
where for $m=0$ or $n=0$ the list of $x$'s or $y$'s is empty.
Hence 
\begin{align*}
    \langle m,n\rangle\in x\circ s \quad \text{or} \quad \langle m,n\rangle\in x\circ s\circ y \quad \text{or} \quad\langle m,n\rangle\in s\circ y,
\end{align*}
where
$x\in x_{m}\circ\dots\circ x_2\circ x_1$ and  $y\in y_1\circ y_2\circ \dots\circ  y_{n}$.
Consequently, for all grid points $\langle m, n\rangle$,
    $$\langle m, n\rangle\Vdash \alpha,$$
as $s\Vdash \textsc{c}\backslash \alpha\land \textsc{c}\backslash \alpha/\textsc{c}\land \alpha/\textsc{c}$, $x_i\Vdash \textsc{c}\land \textsc{c}\backslash \textsc{c}$ and  $y_i\Vdash \textsc{c}\land \textsc{c}\backslash \textsc{c}$, for all $i$.

We refine this further by showing that for all grid points:
\begin{align*}
        &\langle 2m,2n\rangle\Vdash \alpha_{\textsc{ee}}, \\
        &\langle 2m+1,2n\rangle\Vdash \alpha_{\textsc{oe}},\\
        &\langle 2m+1,2n+1\rangle\Vdash \alpha_{\textsc{oo}},\\
        &\langle 2m,2n+1\rangle\Vdash \alpha_{\textsc{eo}}.
\end{align*}
    \begin{figure}
        \begin{center}
        \begin{tikzpicture}[xscale=2.5,yscale=1.5]\small
            \node (00) at (0,0){$\langle2m,2n\rangle$};
            \node (10) at (1,0){$\langle2m{+}1,2n\rangle$}edge(00);
            \node (20) at (2,0){$\langle2m+2,2n\rangle$}edge(10);
            \node (01) at (0,1){$\langle2m,2n+1\rangle$}edge(00);
            \node (11) at (1,1){$\langle2m{+}1,2n{+}1\rangle$}edge(10)edge(01);
            \node (21) at (2,1){$\langle2m{+}2,2n{+}1\rangle$}edge(20)edge(11);
            \node (02) at (0,2){$\langle2m,2n{+}2\rangle$}edge(01);
            \node (12) at (1,2){$\langle2m{+}1,2n{+}2\rangle$}edge(02)edge(11);
            \node (22) at (2,2){$\langle2m{+}2,2n{+}2\rangle$}edge(12)edge(21);
            \node at (0.5,0.5){(a)};
            \node at (1.5,0.5){(b)};
            \node at (0.5,1.5){(c)};
            \node at (1.5,1.5){(d)};
            \node at (0,2.35){$\vdots$};
            \node at (1,2.35){$\vdots$};
            \node at (2,2.35){$\vdots$};
            \node at (2.5,0){$\cdots$};
            \node at (2.5,1){$\cdots$};
            \node at (2.5,2){$\cdots$};
        \end{tikzpicture}
        \caption{Constructing the full grid of points.}
        \label{fig:grid}
        \end{center}
    \end{figure}
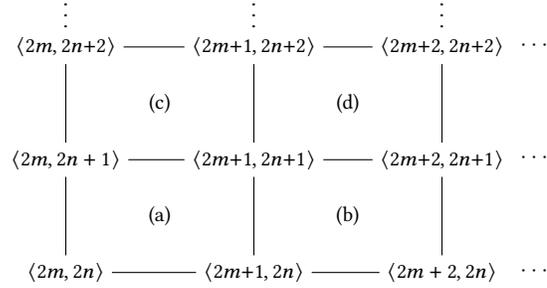
We have already established this for the staircase points, so it remains to extend it to the other points via (a)--(d) (see Fig.~\ref{fig:grid}): 
\begin{itemize}
    \item[(a)] If $\langle 2m, 2n\rangle\Vdash \alpha_{\textsc{ee}}\quad \text{and}\quad \langle 2m+1, 2n+1\rangle \Vdash \alpha_{\textsc{oo}},$\\
    then $\langle 2m+1,2n\rangle\Vdash \alpha_{\textsc{oe}}\quad \text{and} \quad \langle 2m,2n+1\rangle\Vdash \alpha_{\textsc{eo}}.$
    \item[(b)] If $\langle 2m+1, 2n\rangle\Vdash \alpha_{\textsc{oe}}\quad \text{and}\quad \langle 2m+2, 2n+1\rangle \Vdash \alpha_{\textsc{eo}},$\\
    then $\langle 2m+1,2n+1\rangle\Vdash \alpha_{\textsc{oo}}\quad \text{and} \quad \langle 2m+2,2n\rangle\Vdash \alpha_{\textsc{ee}}.$
    \item[(c)] If
    $\langle 2m, 2n+1\rangle\Vdash \alpha_{\textsc{eo}}\quad \text{and}\quad \langle 2m+1, 2n+2\rangle \Vdash \alpha_{\textsc{oe}},$\\
    then $\langle 2m,2n+2\rangle\Vdash \alpha_{\textsc{ee}}\quad \text{and} \quad \langle 2m+1,2n+1\rangle\Vdash \alpha_{\textsc{oo}}.$
    \item[(d)] If
    $\langle 2m+1, 2n+1\rangle\Vdash \alpha_{\textsc{oo}}\quad \text{and}\quad \langle 2m+2, 2n+2\rangle \Vdash \alpha_{\textsc{ee}},$\\
    then $\langle 2m+2,2n+1\rangle\Vdash \alpha_{\textsc{eo}}\quad \text{and} \quad \langle 2m+1,2n+2\rangle\Vdash \alpha_{\textsc{oe}}.$
\end{itemize}

As the proofs are analogous, we only prove (a)---actually only the first conjunct of its consequent. So suppose the antecedent of (a) holds for arbitrary $m,n$. Note that $\langle 2m, 2n\rangle\Vdash \alpha_{\textsc{ee}}$ implies that for some $t\in\mathcal{W}$, 
    $$\langle 2m, 2n\rangle\Vdash \textsc{x}\backslash \Big[\textsc{e}\textsc{e}\lor \big(\textsc{o}\textsc{e}\land \bigvee_{t'\in\mathcal{W}}^{\text{R}(t)=\text{L}(t')}t'\big)\Big]$$
Since also $\langle 2m+1, 2n\rangle\in x_{2m+1}\circ\langle 2m, 2n\rangle$ and $x_{2m+1}\Vdash \textsc{x}$, we have
    $$\langle 2m+1, 2n\rangle\Vdash \textsc{e}\textsc{e}\lor \textsc{o}\textsc{e}.$$
Now assume for contradiction that $\langle 2m+1, 2n\rangle\Vdash \textsc{e}\textsc{e}.$ Since $\langle 2m+1, 2n\rangle\Vdash \alpha$, it follows that $\langle 2m+1, 2n\rangle\Vdash \alpha_{\textsc{ee}}$. Hence $\langle 2m+1, 2n+1\rangle\in\langle 2m+1, 2n\rangle\circ  y_{2n+1}$ and $y_{2n+1}\Vdash \textsc{y}$ imply
    $$\langle 2m+1, 2n+1\rangle\Vdash \textsc{e}\textsc{e}\lor \textsc{e}\textsc{o}.$$
This is a contradiction, as $\langle 2m+1, 2n+1\rangle \Vdash \alpha_{\textsc{oo}}$ by assumption, hence $\langle 2m+1, 2n+1\rangle \Vdash \neg \textsc{e}\textsc{e}\land \neg \textsc{e}\textsc{o}$. Thus, $\langle 2m+1, 2n\rangle\nVdash \textsc{e}\textsc{e}$ whence
    $$\langle 2m+1, 2n\rangle\Vdash \textsc{o}\textsc{e},$$
which together with $\langle 2m+1, 2n\rangle\Vdash \alpha$ implies
    $$\langle 2m+1, 2n\rangle\Vdash \alpha_{\textsc{oe}},$$
as desired.

Finally, note that for all grid points $\langle m,n\rangle$, since $\langle m,n\rangle \Vdash \alpha$, there is a  $t\in \mathcal{W}$ with $\langle m, n\rangle \Vdash t$. Moreover, by the definition of the tile formulas (in particular the negations involved there), this $t$ is unique, as $\langle m,n\rangle\Vdash t$ implies $\langle m,n\rangle\not\Vdash t'$ for every $t'$ distinct from $t$. 

Thus, the function $\tau:\mathbb{N}^2\to \mathcal{W}$ given by 
    $$\tau: (m,n)\mapsto  t\in \mathcal{W} \text{ where }\langle m, n\rangle \Vdash t$$
is well-defined. To demonstrate that $\tau$ is a tiling, consider $u_{t_1}, r_{t_4}$ for a tile $t$ and an arbitrary $\langle m',n'\rangle\Vdash t$; we will establish that $\langle m',n'+1\rangle\Vdash d_{t_1}$ and $\langle m'+1,n'\rangle\Vdash l_{t_4}$. We prove only the latter, as the former is analogous. Without loss of generality we assume that $m'$ and $n'$ are even, so $m'=2m,n'=2n$. 
Then $\langle 2m, 2n\rangle\Vdash \alpha_{\textsc{ee}}$. By the observation above, at most one tile holds at $\langle 2m,2n\rangle$, so 
 $$\langle 2m, 2n\rangle\Vdash \textsc{x}\backslash \Big[\textsc{e}\textsc{e}\lor \big(\textsc{o}\textsc{e}\land \bigvee_{t'\in\mathcal{W}}^{\text{R}(t)=\text{L}(t')}t'\big)\Big],$$
i.e., the above is the disjunct of $\bigvee_{t\in\mathcal{W}}(t\land\cdots)$ in $\alpha_{\textsc{ee}}$ corresponding to~$t$. So, as $x_{2m+1}\Vdash \textsc{x}$, $\langle 2m+1,2n\rangle\in x_{2m+1}\circ\langle 2m, 2n\rangle$ and $\langle 2m+1,2n\rangle\Vdash \alpha_{\textsc{oe}}$---hence $\langle 2m+1,2n\rangle\Vdash \neg \textsc{e}\textsc{e}$---we must have 
    $$\langle 2m+1,2n\rangle\Vdash t'$$
for a $t'\in\mathcal{W}$ such that R$(t)=$ L$(t')$, exactly as required.
\end{proof}

\begin{remark}
The observant reader might be wondering why 
there are two additional occurrences of $\textsc{ee}$ inside the disjunction in
the formula $\alpha_{\textsc{ee}}$, especially as they seem to unnecessarily complicate the argument (in two places). We clarify that they could be omitted for the purposes of Lemma~\ref{l: backward2}, but they will be needed for the statement of Lemma~\ref{l: forward2} to hold.
\end{remark}

\begin{lemma}\label{l: forward2}
Let $\mathcal{W}$ be a finite set of tiles. If $\mathcal{W}$ tiles $\mathbb{N}^2$, then $(\mathcal{P}_{\omega}(\mathbb{N}), \cup, {}^c)\nvDash \phi^{}_\mathcal{W}$.
\end{lemma}
\begin{proof}
     For a tiling $\tau:\mathbb{N}^2\to \mathcal{W}$, recall the  frame $(\mathcal{P}_{\omega}(\mathbb{N}), \cup, {}^c)$ from Remark~\ref{r: fincofindual}, where $\mathcal{P}_{\omega}(\mathbb{N})$ is the set of finite sets of natural numbers, $\circ$ is defined by $\cup$ and $N$ is the completementation operation $^c$, so $u\in s\circ t \iff u=s\cup t$, and $s\Vdash \neg \varphi \iff s\nVdash \varphi$. Writing $2\mathbb{N}=\{2n\mid n\in \mathbb{N}\}$ for the set of even numbers, and $2\mathbb{N}+1=\{2n+1\mid n\in \mathbb{N}\}$ for the set of odd numbers, we define a valuation $V$ as follows:
\begin{align*}
        V(p)&=\varnothing,\\
        V(\textsc{x})&=\{\{n\}\mid n\in 2\mathbb{N}\},\\
        V(\textsc{y})&=\{\{n\}\mid n\in 2\mathbb{N}+1\},\\
        V(\textsc{c})&=\{X\mid X\neq \varnothing \},\\
        V(\textsc{e}\textsc{e})&=\{X\cup Y\mid X\subseteq 2\mathbb{N}, Y\subseteq 2\mathbb{N}+1,   |X| \text{ is even}, |Y| \text{ is even}\},\\
        V(\textsc{o}\textsc{e})&=\{X\cup Y\mid X\subseteq 2\mathbb{N}, Y\subseteq 2\mathbb{N}+1,   |X| \text{ is odd}, |Y| \text{ is even}\},\\
        V(\textsc{o}\textsc{o})&=\{X\cup Y\mid X\subseteq 2\mathbb{N}, Y\subseteq 2\mathbb{N}+1,   |X| \text{ is odd}, |Y| \text{ is odd}\},\\
        V(\textsc{e}\textsc{o})&=\{X\cup Y\mid X\subseteq 2\mathbb{N}, Y\subseteq 2\mathbb{N}+1,   |X| \text{ is even}, |Y| \text{ is odd}\},\\
        V(u_i)&=\{X\cup Y\mid X\subseteq 2\mathbb{N}, Y\subseteq 2\mathbb{N}+1,\text{U}(\tau(|X|,|Y|))=i\},\\
        V(d_i)&=\{X\cup Y\mid X\subseteq 2\mathbb{N}, Y\subseteq 2\mathbb{N}+1,\text{D}(\tau(|X|,|Y|))=i\},\\
        V(l_i)&=\{X\cup Y\mid X\subseteq 2\mathbb{N}, Y\subseteq 2\mathbb{N}+1,\text{L}(\tau(|X|,|Y|))=i\},\\
        V(r_i)&=\{X\cup Y\mid X\subseteq 2\mathbb{N}, Y\subseteq 2\mathbb{N}+1,\text{R}(\tau(|X|,|Y|))=i\}.
\end{align*}
We advise the reader that even and odd numbers are employed in two different ways in this argument. They are used to encode, respectively, the first and second coordinates in the definitions of 
$V(\textsc{e}\textsc{e})$, $V(\textsc{o}\textsc{e})$, 
$V(\textsc{o}\textsc{o})$, and
$V(\textsc{e}\textsc{o})$; the parity of the cardinality of the even/odd elements in a finite set is then used to determine whether the corresponding coordinate is even or odd.
On the other hand, the intuition behind the definitions of $V(u_i)$, $V(d_i)$, $V(l_i)$, and $V(r_i)$ is to relate a finite set---uniquely partitioned as $X\cup Y$ such that $X\subseteq 2\mathbb{N}$ and $Y\subseteq 2\mathbb{N}+1$---to the tile $\tau(|X|,|Y|)$.

We proceed to show that
    \begin{align*}
        (\mathcal{P}_{\omega}(\mathbb{N}), \cup, {}^c, V), \varnothing \nVdash \phi^{}_\mathcal{W}.
    \end{align*}
    Since
$\phi^{}_\mathcal{W} \mathrel{:=}  \Big[\neg (\textsc{x}'\backslash\textsc{o}\textsc{e}^c) \land \textsc{c}\backslash \alpha\land \textsc{c}\backslash \alpha/\textsc{c}\land \alpha/\textsc{c}\Big]\backslash p$, and since
    \begin{align*}
        \varnothing\cup \varnothing=\varnothing\nVdash p,
    \end{align*}
it is enough to show that    
    \begin{align*}
        (\mathcal{P}_{\omega}(\mathbb{N}), \cup, {}^c, V), \varnothing \Vdash \neg (\textsc{x}'\backslash \textsc{o}\textsc{e}^c) \land \textsc{c}\backslash \alpha\land \textsc{c}\backslash \alpha/\textsc{c}\land \alpha/\textsc{c}.
    \end{align*}
Establishing the first conjunct amounts to showing that 
\begin{align*}
        (\mathcal{P}_{\omega}(\mathbb{N}), \cup, {}^c, V), \varnothing \nVdash \textsc{x}'\backslash \textsc{o}\textsc{e}^c.
    \end{align*}
    Since
$\textsc{x}'\mathrel{:=}\textsc{x}\land \textsc{c}\land \textsc{c}\backslash \textsc{c}$, it is enough to observe that
    $$\text{$\{0\}\Vdash \textsc{x}'$ but $\{0\}\cup \varnothing =\{0\}\nVdash \textsc{o}\textsc{e}^c$.}$$
For the first claim, observe that $\textsc{c}$ holds on any non-empty set by the definition of $V(\textsc{c})$.
Moreover, by definition, $\{0\}\Vdash \textsc{c}\backslash \textsc{c}$ iff for every $x,z\in\mathcal{P}_\omega(\mathbb{N})$: $x\Vdash\textsc{c}$ and $z=x\cup \{0\}$ implies $z\Vdash\textsc{c}$. The latter obviously holds since $z$ is non-empty. The second claim follows from the observation that $|\{0\}\cap 2\mathbb{N}|=|\{0\}|=1$ is odd and $|\{0\}\cap (2\mathbb{N}+1)|=|\varnothing|=0$ is even.

It remains to establish
\begin{align*}
        (\mathcal{P}_{\omega}(\mathbb{N}), \cup, {}^c, V), \varnothing \Vdash \textsc{c}\backslash \alpha \land \textsc{c}\backslash \alpha/\textsc{c}\land \alpha/\textsc{c}.
    \end{align*}
By the semantic clauses for $\backslash$ and $/$ (recall Definition~\ref{def-satisfaction}), coupled with the definition of $V(\textsc{c})$, it suffices to show $X\Vdash \alpha$ for each finite non-empty set $X$ of natural numbers. To this end, note that any (finite non-empty) set of natural numbers $X$ is uniquely partitioned into its even and odd parts: $X=X_e\cup X_o$, where $X_e\mathrel{:=}X\cap 2\mathbb{N}$ and $X_o\mathrel{:=}X\cap (2\mathbb{N}+1)$. Without loss of generality, we assume that both $|X_e|$ and  $|X_o|$ are even. By the valuations on these variables,
    $$X\Vdash \textsc{e}\textsc{e} \land \neg\textsc{oe}\land\neg\textsc{oo}\land\neg \textsc{eo}.$$
We proceed to establish the remaining conjuncts in $\alpha_{\textsc{ee}}$.
To obtain $X\Vdash \neg(\textsc{x}'\backslash \textsc{o}\textsc{e}^c)$, 
we consider an even number $n$ such that $n\notin X$, and argue as we did with $\{0\}$ above.
Now it remains to show
\begin{align*}    X\Vdash \bigvee_{t\in \mathcal{W}}\bigg(t&\land \textsc{x}\backslash \Big[\textsc{e}\textsc{e}\lor \big(\textsc{o}\textsc{e}\land \bigvee_{t'\in\mathcal{W}}^{\text{R}(t)=\text{L}(t')}t'\big)\Big]\land \Big[\textsc{e}\textsc{e}\lor \big(\textsc{e}\textsc{o}\land \bigvee_{t''\in\mathcal{W}}^{\text{U}(t)=\text{D}(t'')}t''\big)\Big]/\textsc{y}\bigg).
\end{align*}
Since $\tau$ is a tiling function, and $X$ is uniquely partitioned into its even part $X_e$ and odd part $X_o$, we have that $X\Vdash t$ where $t$ is (the formula corresponding to) the tile $\tau(|X_e|, |X_o|)$.
We will now establish the middle conjunct
\[
X\Vdash
\textsc{x}\backslash \Big[\textsc{e}\textsc{e}\lor \big(\textsc{o}\textsc{e}\land \bigvee_{t'\in\mathcal{W}}^{\text{R}(t)=\text{L}(t')}t'\big)\Big]
\]
(the argument for the right conjunct is analogous). By the  semantic clause for $\backslash$, it suffices to show for all $n\in\mathbb{N}$:
\[
    X\cup\{2n\}\Vdash \textsc{e}\textsc{e}\lor \big(\textsc{o}\textsc{e}\land \bigvee_{t'\in\mathcal{W}, \text{R}(t)=\text{L}(t')}t'\big)
\]
By cases, if $\{2n\}\in X$, then $X\cup\{2n\}=X\Vdash \textsc{e}\textsc{e}$. \\
Also, if $\{2n\}\notin X$, then $|X_e\cup \{2n\}|=|X_e|+1$ is odd, so 
    $$X\cup \{2n\}=(X_e \cup \{2n\})\cup X_o\Vdash \textsc{o}\textsc{e}\land \bigvee_{t'\in\mathcal{W}, \text{R}(t)=\text{L}(t')}t',$$
because $X\cup \{2n\}$ has one even number more than $X$---the cardinality of odd numbers is unchanged---and hence by the valuations on $u_i, d_i, l_i, r_i$, the set $X\cup \{2n\}$ is related to the tile $\tau(|X_e|+1, |X_o|)$. So as $\tau$ is a tiling, the right edge of $\tau(|X_e|, |X_o|)$ is equal to the left edge of $\tau(|X_e|+1, |X_o|)$. 
\end{proof}

\section{Refutation of decidability claims in the literature}\label{sec:erroneous}

Galmiche et al.~\cite{GalmicheMP05:mscs} develop semantic tableaux for BI and the finite model property and decidability is claimed in the abstract. However, the $\text{TBI}'$ tableaux system used to establish the latter is incomplete. In particular, by inspection, the BI-provable formula $p\land (q\lor r)\imp (p\land q)\lor (p\land r)$ has no closed tableau in $\text{TBI}'$. Consequently, the decidability claim is not supported by the given system. Even if this incompleteness were resolved, we concur that ``the claim in the abstract that decidability is obtained is not substantiated''~\cite{Pym-weblink-2026}.

The two other published claims of decidability use the \textit{bunched sequent calculus}~$\LBIold$ that is built from sequents $X\DC A$ where the succedent~$A$ is a BI-formula and the antecedent~$X$ is a \textit{bunch} (see Definition~\ref{d: fbs}); in particular, they make use of the fact that the calculus enjoys cut elimination (e.g., see \cite{GalatosJ}) and investigate backward proof search.
The rules of~$\LBIold$ are similar enough to the original calculus for BI (Figure~\ref{fig-LBI-calculus}) for the present discussion.

Kaminski and Francez~\cite{KaminskiFrancez16} claim that the depth of a bunch---essentially the maximum number of commas along a branch in its grammar tree---in a sequent in proof search can be bounded by the total number of multiplicative connective occurrences ($\ast$ and~$\mimp$) in the endsequent. 
However, an inspection of the crucial lemma~\cite[Lemma 53]{KaminskiFrancez16} reveals a gap: their argument for this bound is that every comma along the branch is generated by a ${\ast}_L$ or ${\mimp}_R$ rule. However, this does not rule out the possibility (see further below for an example) that multiple commas might originate from the same multiplicative occurrence through judicious use of contraction. An even simpler violation of their claim follows from the observation that
the definition of depth (denoted below by $d$) is sensitive to the parenthetical ordering of the semicolons and commas, and hence the depth can increase from conclusion to premise due to the associativity of comma.
E.g. suppose that~$\depth(Z)> \depth(X)$ and~$\depth(Z)> \depth(Y)$.
Then
$\depth((X,Y),Z)=\max(\max( \depth(X), \depth(Y) ) + 1, \depth(Z) ) + 1=\depth(Z) + 1 <
\depth(Z) + 2= \max( \depth(X) , \max( \depth(Y), \depth(Z) ) + 1) + 1  = \depth(X,(Y,Z))$.

Galatos and Jipsen~\cite{GalatosJ} define a directed graph from each bunched sequent. It is claimed that the 
\textit{multiplicative length}---the maximum (taken over all directed paths) of the number of~$\ast$'s and commas in negative position and~$\sepimp$ in positive position on a directed path---is non-increasing from the (directed graph of the) conclusion to the premise(s). For the special case of their argument dealing with BI, they treat comma as an $n$-ary connective rather than a binary connective, for otherwise the associative rule for comma (as-c) would violate their claim e.g. $((p,q),r),s\DC t$ would have greater multiplicative length than $(p,q),(r,s)\DC t$. However, treating comma as an $n$-ary connective turns out to be problematic as well: for the rule instance below left (written in the notation of this paper), the corresponding directed graphs for the conclusion and premise are given below center and right, respectively. The multiplicative length for the premise is~$2$ and for the conclusion it is~$1$, violating the claim.
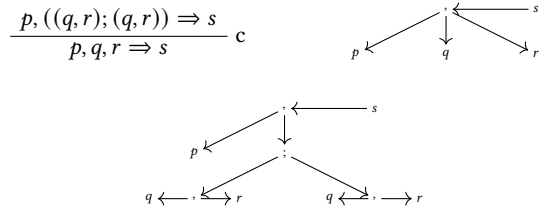
\begin{center}
\begin{tabular}{c}
\AxiomC{$p,((q,r);(q,r))\DC s$}
\RightLabel{c}
\UnaryInfC{$p,q,r\DC s$}
\DisplayProof
\hspace{1cm}
\begin{minipage}{3cm}
\begin{center}
\begin{tikzpicture}[xscale=0.6, yscale=0.6, transform shape]
\node (comma) at (10,10) {$,$};
\node (p) at (8,9) {$p$};
\node (q) at (10,9) {$q$};
\node (r) at (12,9) {$r$};
\node (s) at (12,10) {$s$};

\draw[->] (comma) edge (p) edge (q) edge (r);
\draw[->] (s) edge (comma);
\end{tikzpicture}
\end{center}
\end{minipage}
\\
\\
\begin{minipage}{4cm}
\begin{center}
\begin{tikzpicture}[xscale=0.6, yscale=0.6, transform shape]
\node (comma) at (10,10) {$,$};
\node (p) at (8,9) {$p$};
\node (sc) at (10,9) {$;$};
\node (c1) at (8,8) {$,$};
\node (q1) at (7,8) {$q$};
\node (r1) at (9,8) {$r$};
\node (c2) at (12,8) {$,$};
\node (q2) at (11,8) {$q$};
\node (r2) at (13,8) {$r$};
\node (s) at (12,10) {$s$};

\draw[->] (comma) edge (p) edge (sc);
\draw[->] (sc) edge (c1) edge (c2);
\draw[->] (c1) edge (q1) edge (r1);
\draw[->] (c2) edge (q2) edge (r2);
\draw[->] (s) edge (comma);
\end{tikzpicture}
\end{center}
\end{minipage}
\end{tabular}
\end{center}

Reflecting on these attempts to decide BI,
we observe that the general idea is to control the number of multiplicative connectives/commas on a branch.
The following example shows that we cannot expect this to hold in BI. For brevity, only the premise of interest has been written down in the case of the left implication rules.
Also $x := (s  \mimp  1)$ (utilized in the middle ${\mimp}_L$ application).

\begin{center}
\AxiomC{$\cdots$}
\noLine
\UnaryInfC{$x,(x,(a \mimp ((x \mimp q) \imp c)))  \DC q$}
\RightLabel{${\mimp}_R$}
\UnaryInfC{$(x,(a \mimp ((x \mimp q) \imp c)) )  \DC x \mimp q$}
\RightLabel{${\imp}_L$}
\UnaryInfC{$((x \mimp q) \imp c) ; (x,(a \mimp ((x \mimp q) \imp c)) )  \DC q$}
\RightLabel{${\mimp}_L$}
\UnaryInfC{$(a \mimp ((x \mimp q) \imp c)) ; (x, (a \mimp ((x \mimp q) \imp c)) )  \DC q$}
\RightLabel{$1_L$}
\UnaryInfC{$(1,(a \mimp ((x \mimp q) \imp c)) ) ; (x,(a \mimp ((x \mimp q) \imp c)) )  \DC q$}
\RightLabel{${\mimp}_L$ 
}
\UnaryInfC{$(x,(a \mimp ((x \mimp q) \imp c)) ) ; (x,(a \mimp ((x \mimp q) \imp c)))  \DC q$}
\RightLabel{c}
\UnaryInfC{$x,(a \mimp ((x \mimp q) \imp c))  \DC q$}
\RightLabel{${\mimp}_R$}
\UnaryInfC{$a \mimp ((x \mimp q) \imp c)\DC x \mimp q$}
\RightLabel{${\imp}_L$}
\UnaryInfC{$(a \mimp ((x \mimp q) \imp c)) ; ((x \mimp q) \imp c)  \DC q$}
\RightLabel{${\mimp}_L$}
\UnaryInfC{$(a \mimp ((x \mimp q) \imp c)) ; (a \mimp ((x \mimp q) \imp c)) \DC q$}
\RightLabel{c}
\UnaryInfC{$a \mimp ((x \mimp q) \imp c) \DC q$}
\DisplayProof
\end{center}
By repeated use of this deduction in backward proof search, a sequent with a bunch
$x,(x,(\ldots,(x,(a \mimp ((x \mimp q) \imp c)))\ldots))$ containing a branch with arbitrarily many commas can be obtained.

\section{Proofs of Lemma~\ref{l: dual} and Lemma~\ref{l: dualBI}}\label{sec:duality}
The proof of Lemma~\ref{l: dual} follows standard arguments from (Priestley-style) duality theory (see, e.g., \cite{Urquhart1996}). For the undecidability of BI alone, the reader familiar with UDMF semantics can use Lemma~\ref{l: dualBI} instead.

\setcounter{savedcounter}{\value{theorem}}
\setcounter{savedsection}{\value{section}}

\setcounter{section}{3}
\setcounter{theorem}{3}

\begin{lemma}[Restated]
      If a formula in the language $\{\land, \lor, \neg, \backslash, \slash\}$ is refuted by a disjointive distributive residuated lattice, then it fails in a disjointive associative frame.
\end{lemma}

\setcounter{theorem}{\value{savedcounter}}
\setcounter{section}{\value{savedsection}}

\begin{proof}
    For  a disjointive distributive residuated lattice $\mathbf{A}=(A, \land,$ $\lor, \top, \bot, \neg, \cdot, \backslash, /)$, consider the triple $(X_\mathbf{A}, \circ, N)$, where
    \begin{itemize}
        \item $X_\mathbf{A}$ is the set of prime filters of the lattice reduct of $\mathbf{A}$.\footnote{Recall that for a lattice $(L, \land, \lor)$, a prime filter $x$ is a non-empty 
 proper subset $\varnothing\neq x\subsetneq L$ such that (i) $a\land b\in x$ iff $a\in x$ and $b\in x$, and (ii) $a\lor b\in x$ iff $a\in x$ or $b\in x$.}
        \item $N:\mathcal{P}(X_\mathbf{A})\to \mathcal{P}(X_\mathbf{A})$ is given by \[N(Y)\mathrel{:=}\begin{cases}
            \{x\in X_\mathbf{A}\mid \neg a \in x\} & \text{if $Y=\{x\in X_\mathbf{A}\mid a\in x\}$ for some $a\in A$}\\
            \varnothing & \text{otherwise.}
        \end{cases}
        \]
        \item $\circ :X_\mathbf{A}\times X_\mathbf{A}\to \mathcal{P}(X_\mathbf{A})$ is given by $x\circ y\mathrel{:=}\{z\in X_\mathbf{A}\mid \forall a\in x, b\in y\,{:}\; a\cdot b\in z\}$.
    \end{itemize} 
    For $x,y \in X_\mathbf{A}$, we define  $xy\mathrel{:=}\{a  \cdot b\mid a \in x, b\in y\}$, so $x \circ y=\{z\in X_\mathbf{A}\mid xy\subseteq  z\}$.
    We claim that $(X_\mathbf{A}, \circ, N)$ is a disjointive associative frame. First, note that $N$ is well-defined, since the function $A\ni a\mapsto \{x\in X_\mathbf{A}\mid a\in x\}$ is injective by the prime filter theorem, i.e., $\{x\in X_\mathbf{A}\mid a\in x\}=\{x\in X_\mathbf{A}\mid b\in x\}$ implies $a=b$. Second, to see that $N$ is a disjointive operation, simply observe that $\{x\in X_\mathbf{A}\mid a\in x\}\cap \{x\in X_\mathbf{A}\mid \neg a\in x\}=\varnothing$: if $a, \neg a \in  x\in X_\mathbf{A}$, then $\bot=a\land \neg a \in x$, contradicting the fact that $x$ is proper. 
    
    Third and last, we show that $(x\circ y)\circ z= x \circ (y\circ z)$. Whenever $w\in (x\circ y)\circ z$, i.e., if there is a prime filter $u$ such that $u \in x\circ y$ and $w\in u\circ z$ (i.e., such that $xy \subseteq u$ and $uz \subseteq w$), we have to show that $w\in x\circ (y\circ z)$; that is, that there is a prime filter $v$ such that $w\in x\circ v$ and $v\in y\circ z$ (i.e., such that $xv \subseteq w$ and $yz \subseteq v$). Note that for a filter $v$ the condition $yz \subseteq v$ is equivalent to $\langle yz \rangle \subseteq v$, where $\langle yz \rangle$ denotes the filter generated by $yz$;
    accordingly, given $w\in (x\circ y)\circ z$, we consider the set 
        $$F\mathrel{:=}\{v' \text{ is a proper filter on $\mathbf{A}$} \mid xv'\subseteq w\text{ and } yz\subseteq v'\}$$
    and will apply Zorn's lemma to obtain the desired prime filter $v$. To show that $F$ is not empty, we argue that the filter $\langle yz \rangle$ is  proper and that  $x\langle yz \rangle \subseteq w$.
    It is proper, as otherwise $\bot \in\langle yz \rangle$ would imply that $\bot$ would be above the meet of some elements of $yz$, i.e., that there were $b_1, \hdots, b_n\in y, c_1, \hdots, c_n\in z$ with $\bigwedge_{1\leq i\leq n}b_i\cdot c_i=\bot$. This would lead to the contradiction $\bot \in w$, since $\left(\top\cdot \bigwedge_{1\leq i\leq n} b_i\right)\in xy \subseteq u$ and
    \begin{align*}
    w&\supseteq uz\ni \left(\top\cdot \bigwedge_{1\leq i\leq n} b_i\right)\cdot \bigwedge_{1\leq i\leq n}c_i\overset{\text{(as.)}}{=}\top\cdot \left(\bigwedge_{1\leq i\leq n} b_i\cdot \bigwedge_{1\leq i\leq n}c_i\right)\overset{\text{(mon.)}}{\leq} \top\cdot \bigwedge_{1\leq i\leq n}b_i\cdot c_i=\top \cdot \bot=\bot,
    \end{align*}
    where $^{\text{(mon.)}}$ refers to monotonicity of $\cdot$, which together with the last equality follow by the residuation property.
    To see that $x\langle yz \rangle \subseteq w$, note that  $x yz  \subseteq uz \subseteq w$, so  $x\langle yz \rangle \subseteq \langle xyz \rangle \subseteq \langle w\rangle = w$. 
    
    Thus $F$ is non-empty. So, as $\bigcup v'_i\in F$ for every chain $v'_1\subseteq v_2'\subseteq \cdots$ of elements of $F$, it follows by Zorn's lemma that the poset $(F, \subseteq)$ has a maximal element $v\in F$. As $v\in F$, it remains to show that $v$ is prime; we assume $d\notin v$  and $e\notin v$, and we will show $d \jn e\notin v$.  
    We have $yz \subseteq v \subseteq \langle v\cup\{d\}\rangle$ and, by maximality, $\langle v\cup\{d\}\rangle \notin F$, so $x\langle v\cup\{d\}\rangle \not \subseteq w$ or $\langle v\cup\{d\}\rangle$ is  not proper, in which case $\bot\in \langle v\cup\{d\}\rangle$ hence $\bot \in x \{\bot\} \subseteq x\langle v\cup\{d\}\rangle$; since $\bot \notin w$, in both cases we get $x\langle v\cup\{d\}\rangle \not \subseteq w$, i.e., there are $d_x\in x$ and $d_v\in v$ s.t. $d_x\cdot (d_v\land d)\notin w$. Likewise, there are $e_x\in x$ and $e_v\in v$ s.t. $e_x\cdot (e_v\land e)\notin w$.
    So by additivity of $\cdot$ and primeness of $w$,
    \begin{align*}
        (d_x\land e_x)\cdot [(d_v\land e_v)\land (d\lor e)]&= (d_x\land e_x)\cdot [(d_v\land e_v)\land d]\lor (d_x\land e_x)\cdot [(d_v\land e_v)\land e]\\
        &\leq d_x\cdot (d_v\land d)\lor e_x\cdot (e_v\land e)\notin w.    \end{align*}
    So as $d_x\land e_x\in x$ and $v\in F$, we have $(d_v\land e_v)\land (d\lor e)\notin v$, whence as $d_v\land e_v\in v$, we have $d\lor e\notin v$, as required. This proves $(x\circ y)\circ z\subseteq x \circ (y\circ z)$. The converse is proven analogously. We therefore conclude that $(X_\mathbf{A}, \circ, N)$ is a disjointive associative frame.

    Now let $\varphi$ be a formula in the language $\{\land, \lor, \neg, \backslash, \slash\}$, and assume $\mathbf{A}\nvDash \varphi$, i.e., that there is a homomorphism $h$ from the formula algebra to $\mathbf{A}$ such that $h(\varphi)\neq \top$. Then $\mathfrak{M}=(X_\mathbf{A}, \circ, N, V_h)$ is a disjointive associative model where
        $$V_h(p)\mathrel{:=}\{x\in X_\mathbf{A}\mid h(p)\in x\}.$$
An induction shows that for all $x\in X_\mathbf{A}$ and formulas $\alpha$ in the language $\{\land, \lor, \neg, \backslash, \slash\}$,
\[
    \|\alpha\|=\{x\in X_\mathbf{A}\mid h(\alpha)\in x\}.
\]
The base case is by definition; $\lor, \land$ follow by defining properties of prime filters; and $\neg$ is by definition, as $\|\neg \alpha\|=N(\|\alpha\|)\overset{\text{IH}}{=}N(\{x\in X_\mathbf{A}\mid h(\alpha)\in x\})=\{x\in X_\mathbf{A}\mid \neg h(\alpha) \in x\}=\{x\in X_\mathbf{A}\mid h(\neg\alpha) \in x\}$. Lastly, `$\supseteq$' of the inductive step for the residuals follow easily, but the converse is less trivial, so we cover it here for $\backslash$.

Accordingly, suppose by way of contraposition 
that $h(\alpha)\backslash h(\beta)=h(\alpha\backslash\beta)\notin y$; we show that $y\nVdash \alpha\backslash \beta$, i.e., we find prime filters $x,z$ such that $z\in x\circ y$, $x\Vdash \alpha$ and $z\nVdash \beta$. We define the filter
$$z_0\mathrel{:=}\langle h(\alpha)y \rangle=\{a\in A\mid \exists b\in y\,{:}\; a\geq h(\alpha)\cdot b\}.$$ Note that $h(\beta)\notin z_0$, because otherwise we would have $h(\alpha)\cdot b\leq h(\beta)$ for some $b\in y$, whence by residuation $y\ni b\leq h(\alpha)\backslash h(\beta)$, which would contradict $h(\alpha)\backslash h(\beta)\notin y$. So by the prime filter theorem, there is $z\in X_\mathbf{A}$ such that $z_0\subseteq z\not\ni h(\beta)$, whence by induction hypothesis $z\nVdash \beta$. Next, observe that $x_0\mathrel{:=}\{a\in A\mid a\geq h(\alpha)\}$ is a proper filter: it is clearly a filter, and it is proper because (i) $\top\leq \bot\backslash h(\beta)$ holds by residuation and (ii) $h(\alpha)\backslash h(\beta)\notin y$. Hence by the same Zorn's lemma argument as before, we get a prime filter $x\supseteq x_0$ such that $x\circ y\ni z$. Consequently, as $h(\alpha)\in x_0\subseteq x$ entails by induction hypothesis that $x\Vdash \alpha$, we have $y\nVdash \alpha\backslash \beta$, as desired.

Finally, since $h(\varphi)\neq \top$, there is a prime filter $x$ s.t. $h(\varphi)\notin x$, whence $x\nVdash \varphi$ so $\mathfrak{M}\nvDash \varphi$, completing the proof of the lemma.
\end{proof}
\begin{remark}
     The use of choice principles in the prior argument (used to construct prime filters from proper filters) may be a cause of concern to readers who might then wonder whether the undecidability of BI (and, more generally, Theorem~\ref{t: main}) could be independent of ZF. Such worries can be defused. In a nutshell, the situation here is all the same as that of ZF and first-order logic: ZF proves completeness and compactness of first-order logic for \textit{countable} languages (as they come well-ordered without choice), but not for \textit{arbitrary} languages. 
     
     In more detail, if a formula in the language $\{\land, \lor, \neg, \backslash, \slash\}$ is refuted by a disjointive distributive residuated lattice, then it is, in particular, refuted by the Lindenbaum-Tarski algebra over our language. As our language can be well-ordered without choice, so can this algebra. Hence the extensions of proper filters to prime filters in the previous lemma can---for this specific algebra---be done choice-free via the usual recursive construction of Lindenbaum's lemma (the latter is possible precisely because we have a well-order). 
     In other words, choice is only needed for representation à la Stone, not for completeness.
\end{remark}

\setcounter{savedcounter}{\value{theorem}}
\setcounter{savedsection}{\value{section}}

\setcounter{section}{3}
\setcounter{theorem}{4}

\begin{lemma}[Restated]
    If a formula in the language $\{\land, \lor, \neg, \mimp\}$ is refuted by a BI-algebra, then it fails in a disjointive associative frame.
\end{lemma}

\setcounter{theorem}{\value{savedcounter}}
\setcounter{section}{\value{savedsection}}

\begin{proof}
    Suppose $\nvdash \top\DC \varphi$. Then by completeness,\footnote{See, e.g., \cite{DochPym2019} on upwards and downwards closed monoidal models and frames. 
    } there is an upwards and downwards closed monoidal model $(S, \preccurlyeq, \circ, E, V)$ and $s\in S$ such that $(S, \preccurlyeq, \circ, E, V), s\nVdash \varphi$. We claim that $\mathfrak{M}\mathrel{:=}(S, \circ, N, V)$, where $N(X)\mathrel{:=}({\downarrow} X)^c\mathrel{:=}\{s\in S\mid \exists s'\in X\text{ s.t. $s\preccurlyeq s'$}\}^c$, is a disjointive associative model. Rather than listing every condition imposed on $(S, \preccurlyeq, \circ, E, V)$, it is enough to note that $\circ: S\times S\to \mathcal{P}(S)$ is associative; $\preccurlyeq$ is reflexive; and the valuation $V$ is, in particular, a function $V:P\to\mathcal{P}(S)$. Further, the clause for $\neg$ within this semantics is the intuitionistic negation $\|\neg\alpha\|\mathrel{:=}({\downarrow}\|\alpha\|)^c$, hence $\|\alpha\|\cap \|\neg\alpha\|=\varnothing$ (because $\preccurlyeq$ is reflexive). It follows that $\mathfrak{M}=(S, \circ, N, V)$ is a disjointive associative model. Consequently, since the clauses for $\land, \lor, \sepimp$ within this BI semantics coincide with ours, we get that $\mathfrak{M}, s\nVdash\varphi$, as desired.
\end{proof}

\begin{acks}
    Knudstorp was supported in part by the MOSAIC project (H2020-MSCA-RISE-2020, 101007627), and in part by the Nothing is Logical (NihiL) project (NWO OC 406.21.CTW.023).
Ramanayake was supported in part by the Austrian Science Fund (FWF) project P33548, and the Dutch Research Council (NWO) project OCENW.M.22.258.

\end{acks}

\bibliographystyle{IEEEtranS}%
\bibliography{sn-bibliography}

\appendix

\section{Reduction from acceptance in ACMs}\label{sec-ACM-reduction}

We establish the undecidability of BI via a reduction from the acceptance problem for And-branching Counter Machines (ACMs). The same problem was used to show the undecidability of propositional linear logic~\cite{LinMitSceSha1992}; meanwhile, the acceptance problem for expansive ACMs was used to establish a non-primitive recursive lower bound for $\m {FL_{ec}}$~\cite{Urquhart1999}, also known as \textbf{LR}. 
The proof here exhibits two novelties in comparison: (i)~a novel encoding is introduced to tame the structurally richer intuitionistic contraction, in order to replicate machine instructions on demand, and (ii)~the completeness argument---i.e., reading off a computation tree from a proof of a BI-theorem that is an instance of the reduction---uses a semantic argument via algebraic residuated frames (allowing us to avoid an involved proof-theoretic case analysis).

This appendix is self-contained.
ACMs are introduced in Section~\ref{subsec-ACMs}. The reduction function, and the simulation of a computation as a proof in BI (soundness of the reduction) is presented in Section~\ref{subsec-reduction-fn}.
A review of residuated frames and an algebraic presentation of ACMs appears in Section~\ref{subsec-review-RF}.
We conclude with the residuated frames argument for soundness (Section~\ref{subsec-soundness}) and completeness (Section~\ref{subsec-completeness}).

\subsection{And-branching counter machines (ACM)}\label{subsec-ACMs}

Following a slight notational adaptation of~\cite{LinMitSceSha1992,Urquhart1999}, a $k$-ACM $M=(\States, R_k,\Inst, q_f)$ is a tuple such that $\States$ is a finite set of states, $R_k=\{r_1,\ldots,r_k\}$ is a set of $k$ propositional variables, $\Inst$ is a finite set of \emph{instructions},
and $q_f\in\States$ is a distinguished final state. 
Every instruction is an \emph{increment} $q\mimp (q' \ast r_i)$, a \emph{decrement} $(q\ast r_i)\mimp q'$, or a fork $q\mimp (q'\lor q'')$, for $\{q,q',q''\}\subseteq \States$ and $r_i\in R_k$.

A \emph{configuration} of $M$ is a formula $\config{q}{R}$ where $q\in Q$ and $R$ is a finite product of variables from $R_k$, taken up to commutativity, or the unit $1$. 
We sometimes write $q1$ as $q$.
The number of occurrences of $r_i$ in $R$ is interpreted as the value of the $i$-th register. 
Later on, we encode a configuration in a sequent as the multiset that is obtained from the occurrences of each propositional variable in $qR$. We write $r^n$ to mean the product of $n$ copies of $r$. There is an obvious bijective translation between these notations.

A \emph{computation tree} is defined in the usual way as a tree with nodes labelled by configurations, such that the labels of the children (zero, one or two, at each node) are those obtained by applying one of the instructions to the label of their parent.
Specifically, if a node labelled $C$ has a single child labelled $C'$ (a fact that we denote $C\rightsquigarrow C'$), then either \begin{description}
\item[(increment)] $C=\config{q}{r_1^{n_1}\cdots r_i^{n_i}\cdots r_\regnum^{n_\regnum}}$, $C'=\config{q'}{r_1^{n_1}\cdots r_i^{n_i+1}\cdots r_{\regnum}^{n_\regnum}}$, for some $q\mimp (q'\ast r_i)\in\Inst$, or

\item[(decrement)] $C=\config{q}{ r_1^{n_1}\cdots r_i^{n_i+1}\cdots
r_\regnum^{n_\regnum}}$, $C'=\config{q'}{r_1^{n_1}\cdots r_i^{n_i}\cdots
r_n^{n_\regnum}}$, for some $(q\ast r_i)\mimp q'\in\Inst$,
\end{description}
and if a node labelled $C$ has two children labelled $C'$ and $C''$ (which we denote $C\rightsquigarrow C'\lor C''$), then
\begin{description}
\item[(fork)] $C=\config{q}{R}$, $C'=\config{q'}{R}$,  $C''=\config{q''}{R}$,  for some  $q\mimp (q'\lor q'')\in\Inst$.
\end{description}

An ACM \emph{accepts} a configuration~$C$ if there is a computation tree such that its root is labelled $C$ and every leaf
is labelled $\config{q_f}{}$.
\begin{theorem}[Lincoln et al, 1992~\cite{LinMitSceSha1992}]
There is a $2$-ACM for which acceptance is undecidable.
\end{theorem}

\subsection{Reduction from ACMs (soundness of the encoding)}\label{subsec-reduction-fn}

Given an ACM~$M=(\States, R_k,\Inst, q_f)$, we define
$$
i:= \bigwedge I, \qquad
t:= (i \mimp q_f)\ra q_f,\qquad
\theta:= (\top \mimp t)\mt 1.$$
Let $\config{q}{R}$ be any configuration. We show the following reduction.
\[
\text{
$M$ accepts $\config{q}{R}$
iff
$q,R,\theta\Ra q_f$ is provable in BI}
\]
Provability is defined using the usual sequent calculus LBI~\cite{Pym02:book,GalmicheMP05:mscs}.
We freely use the cut-elimination theorem for LBI~\cite{Pym02:book,GalatosJ}
which states that every provable sequent has a cut-free proof.

The left-to-right direction (soundness) uses proof rules to simulate machine instructions, and it motivates the definition of $\theta$.
Specifically, let~$\mathcal{T}$ be a computation tree with root labelled $\config{q}{R}$ and every leaf labelled $\config{q_f, 1}$.
Obtain an LBI proof by induction on the height of $\mathcal{T}$. Base case: $\mathcal{T}$ is a single node labelled $\config{q_f}{}$.
\begin{center}
\AxiomC{}
\UnaryInfC{$q_f\Ra q_f$}
\RightLabel{$\equiv$}
\UnaryInfC{$q_f,\varnothing_\times\Ra q_f$}
\RightLabel{$1_L$}
\UnaryInfC{$q_f,1\Ra q_f$}
\RightLabel{w}
\UnaryInfC{$q_f, ((\top\mimp t);1)\Ra q_f$}
\RightLabel{$\land_L$}
\UnaryInfC{$q_f,\theta\Ra q_f$}
\DisplayProof
\end{center}


Inductive case. 
Let $R=r_1^{n_1}\cdots r_\regnum^{n_\regnum}$.
Suppose the last instruction $T_1\mimp T_2$ in $\mathcal{T}$ is a decrement $(q.r_1)\mimp q'$  (the other cases are similar).
Thus, $\config{q}{R}\rightsquigarrow\config{q'}{R'}$ and acceptance of $\config{q'}{R'}$ is witnessed by the strict subtree~$\mathcal{T}_0$.
It follows that $R'=r_1^{n_1-1}r_2^{n_2}\cdots r_\regnum^{n_\regnum}$.
It suffices to obtain a deduction of $q,R,\theta\Ra q_f$ from 
$q',R',\theta\Ra q_f$ (this can be seen as the proof-theoretic gadget implementing $(q.r_1)\mimp q'$).
\begin{center}
\begin{scriptsize}
\AxiomC{$q,R\Ra \top$\hspace{-2cm}}
\AxiomC{$q,r_1\Ra q.r_1$}
\AxiomC{$	q',R',\theta\Ra q_f$}
\dashedLine
\UnaryInfC{$q',r_1^{n_1 - 1},r_2^{n_2},\ldots,r_\regnum^{n_\regnum},\theta\Ra q_f$}
\RightLabel{$\mimp_L$}
\BinaryInfC{$q,r_1^{n_1},\ldots,r_\regnum^{n_\regnum},\theta,(q.r_1)\mimp q'\Ra q_f$}
\RightLabel{repeated $\land_L$, w}
\UnaryInfC{$q,R,\theta, \land\Inst \Ra q_f$}
\RightLabel{$\mimp_R$}
\UnaryInfC{$q,R,\theta\Ra \land\Inst\mimp q_f$}
\AxiomC{$q_f\Ra q_f$}
\RightLabel{$\imp_L$}
\BinaryInfC{$((\land\Inst\mimp q_f)\imp q_f; q,R,\theta\Ra q_f$}
\RightLabel{$\mimp_L$}
\BinaryInfC{$q,R,\top\mimp\left(\left(\land\Inst\mimp q_f\right)\imp q_f\right); q,R,\theta\Ra q_f$}
\RightLabel{$\land_L$}
\UnaryInfC{$q,R,\left(\top\mimp\left(\left(\land\Inst\mimp q_f\right)\imp q_f\right)\right)\land 1; q,R,\theta\Ra q_f$}
\dashedLine
\UnaryInfC{$q,R,\theta; q,R,\theta\Ra q_f$}
\RightLabel{c}
\UnaryInfC{$q,R,\theta\Ra q_f$}
\DisplayProof
\end{scriptsize}
\end{center}
Every leaf in the above is provable; in particular, $q',R',\theta\Ra q_f$ is provable by IH since $\mathcal{T}_0$ has strictly smaller height than $\mathcal{T}$.

For the right-to-left direction (completeness), the proof-theoretic approach must establish that \emph{any} cut-free proof of $q,R,\theta\Ra q_f$ corresponds to a computation tree accepting $\config{q}{R}$. If the proof ends as the deduction above, then we simply extract a computation from provability of $q',R',\theta\Ra q_f$ via IH, and prepend the instruction $(q.r_1)\mimp q'\in\Inst$. However, proofs in BI are more liberal than computations, so there are many other ways for the proof to conclude. Indeed, the situation is more complicated here than for linear logic or FLec, since BI is structurally richer. E.g., the proof may conclude with a sequence of contraction and weakening rules.
\begin{center}
\AxiomC{$X\Ra q_f$}
\noLine
\UnaryInfC{sequence of $w$ and $c$ rules}
\noLine
\UnaryInfC{$q,R,\theta\Ra q_f$}
\DisplayProof
\end{center}
All sequences are not possible here, since e.g., $\Ra q_f$ is not provable. So, what is required is a characterization of the possible bunches~$X$, and a generalized interpretation of acceptance on $X\Ra q_f$. The ensuing complications motivate an elegant alternative, via a semantic interpretation of the sequent in terms of residuated frames.

\subsection{Soundness via (G)BI-algebras}
\label{subsec-soundness}

We prove the undecidability for BI and its non-commutative version GBI at the same time. The latter uses both divisions~$\rd$ and $\ld$. The proof for the BI case can be obtained by uniformly replacing the two divisions by $\mimp$, or by noting that the two divisions are equivalent in the presence of commutativity. The result in fact applies to all logics between BI and GBI, and even further (Corollaries~\ref{c: UndInt} and \ref{c: (G)BI}).

Given a $k$-ACM $\texttt{M}=(Q, R_k, I, q_f)$, write the instructions of $I$ using right division as $d' \rd d$ (instead of $d\mimp d'$), and update the terms defined before as follows, using left division.
$$
i:= \bigwedge I, \qquad
t:= (i \ld q_f)\ra q_f,\qquad
\theta:= (\top \ld t)\mt 1.$$
We write $c \rightsquigarrow c_1 \jn \cdots \jn c_m$ if there is a computation tree with root $c$ and multiset of leaves equal to $\{c_1, \ldots, c_m\}$. Also, we write $\overline{q_f}$ for any finite (non-idempotent) join $q_f \jn \cdots \jn q_f$ of $q_f$'s.

\begin{lemma}\label{l: soundness} If $c \rightsquigarrow \overline{q_f}$ then $\mathsf{GBI} \models c\theta \leq q_f$, for every ACM $\texttt{M}=(Q, R_k, I, q_f)$ and configuration $c$.
\end{lemma}

\begin{proof}
We use induction on the length of the computation; for length equal to zero, we have $c \theta=q_f \theta\leq q_f1=q_f$.  If $c \rightsquigarrow_{n+1} \overline{q_f}$, then there exist
configurations $c_1, \ldots, c_m$, $m \in \{1, 2\}$, such that
$c \rightsquigarrow_1 c':= c_1 \jn \cdots \jn c_m$, and 
$c_i \rightsquigarrow_n \overline{q_f}$ for all $i$. By the induction hypothesis, $c_i\theta \leq q_f$  in BI  for all $i$, so $c'\theta \leq q_f$. Also, since $c \rightsquigarrow_1 c'$,
there exists $d'\rd d\in I$ and  $R \in R_k^*$ such that $c=dR$ and $c'=d'R$, for some  $R \in R_k^*$.
Note that $c=dR \leq i \ld idR \leq i \ld [(d' / d)dR] \leq i \ld d'R=i \ld c'$, hence $c \theta \leq (i \ld c')\theta \leq i \ld c'\theta$. So, 
\begin{align*}c\theta&=c\theta \mt c\theta \leq \top(\top \ld t)\mt c\theta \leq t \mt (i \ld c'\theta)\leq 
[(i \ld q_f)\ra q_f] \mt (i \ld q_f)
\leq q_f. \qedhere
\end{align*}
\end{proof}
The above is an algebraic presentation of the proof-theoretic argument we gave before.
Indeed, the step 
$c\theta=c\theta \mt c\theta$ corresponds to commencing (bottom-up) from $q,R,\theta\Ra q_f$ by an application of contraction. 
As noted, the proof uses two divisions---indeed, observe that both appear in $\theta$---to cover the non-commutative case. The proof specializes to BI by conflating the divisions.

\subsection{A brief review of residuated frames}\label{subsec-review-RF}

A \emph{distributive residuated frame} is a structure $$\m W=(W, W', N, \circ, \1, \ldd, \rdd, \omt, \T, \ldo, \rdo )$$ where $W, W'$ are sets, $\circ$ and $\omt$ are binary operations on $W$, $\1, \T \in W$, $\ldd$ and $\ldo$ are functions from $W \times W'$ to $W'$ and $\rdd, \rdo: W' \times W \ra W'$, $N \subseteq W \times W'$, and for all $x,y,w \in W$ and $z \in W'$, we have the following implications (double lines indicate bi-implications)
{\small
$${
\infer=[(\circ N)]{y \N x \ldd z}{x \circ y \N z} 
\quad 
\infer=[(\circ N)]{x \N z \rdd y}{x \circ y \N z} 
\quad 
\infer=[(\circ a)]{x \circ (y \circ w) \N z}{(x \circ y) \circ z \N z}
}$$
$${\infer[(\circ e)]{y \circ x \N z}{x \circ y \N z} 
\qquad 
\infer=[(\1)]{\1 \circ x \N z}{x \N z}
\qquad 
\infer=[(\1)]{x \circ \1 \N z}{x \N z} 
}$$
$${
\infer=[(\omt N)]{y \N x \ldo z}{x \omt y \N z} 
\quad 
\infer=[(\omt N)]{x \N z \rdo y}{x \omt y \N z} 
\quad
\infer=[(\omt a)]{x \omt (y \omt w) \N z}{(x \omt y) \omt w \N z}} 
$$
$${
\infer[(\omt e)]{y \omt x \N z}{x \omt y \N z} 
\quad 
\infer=[(\T)]{\T \omt x \N z}{x \N z} 
\quad 
\infer[(\omt c)]{x \N z}{x \omt x \N z} 
\quad 
\infer[(\omt i)]{x \omt y \N z}{x \N z} 
}$$
}
without $(\circ e)$; if it also satisfies  $(\circ e)$ it is called \emph{commutative}. The bi-implications $(\circ N)$, $(\omt N)$ are the \emph{nuclear conditions} and the remaining ones are the \emph{structural conditions} of a frame.


For an example from algebra, if $\m A=(A, \mt, \jn, \cdot, 1, \ld, \rd, \ra, \top)$ is a GBI algebra (a BI-algebra), then  $\m W_{\m A}=(A, A, \leq, \cdot, 1, \ld, \rd, \mt, \top, \ra, \leftarrow)$ is a (commutative, respectively) distributive residuated frame. 
For an example from proof theory, $$\m W_{GBI}=(W, W', N, \circ, \1, \ldd, \rdd, \omt, \T, \ldo, \rdo )$$ is a  distributive residuated frame, where 
$W$ is the absolutely free $\{\circ, \1, \omt, \T \}$-algebra over  the set $Fm$ of formulas of $GBI$ (these are the left-hand sides of GBI-sequents, where $\circ$ is usually written by comma and $\omt$ by semicolon), $W'=S_W \times Fm$, where $S_W$ denotes the \emph{sections} over $W$, i.e., elements $u=u(\_)$ of $W$ with a single hole (so if $x \in W$ and $u\in S_W$, then $u(x) \in W$), $x \N (u,a)$ iff the sequent $u(x) \Ra a$ is provable in $BI$, and $x \ldd (u,a):=(u(x \circ  \_),a)$, $(u,a) \rdd  y:=(u(\_ \circ y),a)$, $x \ldo (u,a):=(u(x \omt  \_),a)$, $(u,a) \rdo y:=(u(\_ \omt y),a)$. For the case of BI, the frame is actually commutative.


Now, in any $\m W$, for $X \subseteq W$ and $Y \subseteq W'$, we define 
    $X^\triangleright=\{b \in W' \mid X \N b\}$ and $Y^\triangleleft=\{a \in W \mid a \N  Y\}$,
where $X \N  b$ means $x  \N  b$, for all $x\in X$, and $a \N  Y$ means $a  \N  y$, for all $y\in Y$; moreover,  $X \N  Y$ means $x \N  y$, for all $x \in X$ and $y \in Y$. We write $a^\btr$ for $\{a\}^\btr$ and $b^\btl$ for $\{b\}^\btl$. The maps 
${}^\triangleright$ and ${}^\triangleleft$ form a Galois connection, i.e., for all $X \subseteq W$ and $Y \subseteq W'$ we have $X \subseteq Y^\triangleleft$  iff   $Y \subseteq X^\triangleright$ and the map  $\gamma: X \mapsto X^{\triangleright\triangleleft}$ is a closure operator on the powerset  $\mathcal{P}(W)$. 

Then $\m W^+:=(\gamma[\mathcal{P}(W)], \cap, \cup_\gamma, \circ_\gamma, 1_\gamma, \ld, \rd, \ra, \top)$ is a GBI algebra, where  $X \cup_\gamma Y:= \gamma(X \cup Y)$, $X \circ_\gamma Y:= \gamma(X \circ Y)$, $1_\gamma=\gamma(\1)$,  $X \ld Y=\{w \in W \mid X \circ \{w\} \subseteq Y\}$, $Y \rd X=\{w \in W \mid \{w\} \circ X \subseteq Y\}$, $X \ra Y=\{w \in W \mid X \omt \{w\} \subseteq Y\}$, $\top=W$;  this GBI-algebra is called the \emph{Galois algebra} of $\m W$ and it is actually a BI-algebra when $\m W$ is commutative.

 For each $b \in W'$, the set $b^\btl$ is closed and it is called a \emph{basic closed set}. Also, the basic closed sets form a \emph{basis}: every closed set is an intersection of basic closed sets. As a result, for $X, Y \in W^+$, we have $X \subseteq Y$ iff for all $b \in W'$ we have: $Y \subseteq b^\btl \Ra X\subseteq b^\btl$.

 The map $\gamma$ is actually a nucleus, i.e., it further satisfies $\gamma(x)\gamma(y)\leq \gamma(xy)$, or equivalently $\gamma(\gamma(x)\gamma(y))= \gamma(xy)$. Also, simply because it is a closure operator it also satisfies $\gamma(\gamma(x)\jn \gamma(y))= \gamma(x \jn y)$.


We now provide an \emph{algebraic presentation of ACMs}. For an integer $k$, a  $k$-ACM is a structure $\texttt{M}=(Q, R_k, I, q_f)$, where $Q$ is a set of \emph{states}, $q_f \in Q$, $R_k=\{r_1, \ldots, r_k\}$ is the set of \emph{register tokens}, and $I$ is a set of instructions of the form $q'r\rd q$ (increment), $q' \rd qr$ (decrement) and $(q_1 \jn q_2)\rd q$ (fork), where $q, q', q_1, q_2\in Q$ and $r \in R$.

 A \emph{configuration} is an element of the free commutative monoid $(Q \cup R_k)^*$ generated by $Q \cup R_k$ of the form $q r_1^{n_1}\cdots r_k^{n_k}$, where $q \in Q$; as usual, $r^n$ denotes the $n$-fold product of $r$.

We also consider the free commutative semigroup $(J, \jn)$ over  $(Q \cup R_k)^*$, so the elements of $J$ admit a normal form up to commutativity: $x_1 \jn \ldots \jn x_m$, where $m \in \mathbb{Z}^+$ and $x_1, \ldots, x_m \in (Q \cup R_k)^*$. On $J$ we define the operation $\cdot$ by 
$(\bigvee x_i)\cdot (\bigvee y_j)=\bigvee (x_i\cdot y_j)$, for all $x_i, y_j \in (Q \cup R_k)^*$. Note that $\cdot$ distributes over $\jn$, so  $(J, \jn, \cdot)$ has a semiring structure. Note that we do not assume $\jn$ to be idempotent so  $x_1 \jn \ldots \jn x_m$ is essentially a multiset and not a set. (This is because in the computation tree below, we do not want to identify leaves.)

On $J$ we define the \emph{one-step computation} relation $\rightsquigarrow_1$ to be the smallest relation that includes $d \rightsquigarrow_1 d'$, for all $d' \rd d \in I$, and is closed under multiplication (if $x \rightsquigarrow_1 y$ then $xz \rightsquigarrow_1 yz$, for all $x,y,z \in J$) and join (if $x \rightsquigarrow_1 y$  then $x\jn z \rightsquigarrow_1 y \jn z$, for all $x,y,z \in J$). 
So, for example, if $q'r_1 \rd q\in I$, then $q \rightsquigarrow_1 q'r_1$, but also $qr_1r_2 \rightsquigarrow_1 q'r_1r_1 r_2$ and $qr_1r_2  \jn q'r_1 r_3 \rightsquigarrow_1 q'r_1r_1 r_2 \jn q'r_1 r_3$. 
It is easy to see that since multiplication distributes over join, for every one-step computation $x \rightsquigarrow_1 y$, then there exists  $d'\rd d \in I$ such that $x=dz \jn w$ and $y=d'z \jn w$, for some $z \in (Q \cup R_k)^*$ and $w \in J$ (or $x=dz$ and $y=d'z$). We define the \emph{computation forest} of a one-step computation $dz \jn w \rightsquigarrow_1 d'z \jn w$, where  $w=\bigvee_j w_j$, i.e., of $dz \jn \bigvee_j w_j \rightsquigarrow_1 d'z \jn \bigvee_j w_j$ that has minimal elements $dz$ and all of the $w_j$; also the child node of $dz$ is $d'z$, if the instruction is increment of decrement, and the children nodes of  $dz$ are $q_1z$ and $q_2z$ (the $w_j$ have no children).

We define $\rightsquigarrow$ to be the reflexive transitive closure of $\rightsquigarrow_1$; to be precise $\rightsquigarrow_n$ is the $n$-fold composition/power of $\rightsquigarrow_1$ and   $\rightsquigarrow$ is the union of the $\rightsquigarrow_n$ over all $n \in \mathbb{N}$. A \emph{computation of length $n$} is a sequence of elements such that $x_0 \rightsquigarrow_1 x_1 \rightsquigarrow_1 \ldots \rightsquigarrow_1 x_n$. The computation forest of a computation is defined recursively. Given computation forests $F$ for $x \rightsquigarrow dz \jn w$ and $d' \rd d \in I$, the computation forest of $x \rightsquigarrow dz \jn  w\rightsquigarrow_1 d'z \jn  w$ is obtained by extending $F$ so that: the child node of $dz$ is $d'z$, if the instruction is increment of decrement, and the children nodes of  $dz$ are $q_1z$ and $q_2z$

Note that we took $(Q \cup R_k)^*$ to be the free commutative monoid and, for example, $qr_1r_2=qr_2r_1$, hence $qr_1r_2\rightsquigarrow qr_2r_1$. However, we could have taken  $(Q \cup R_k)^*$ to be the free monoid
and stipulate commutativity only at the level of the computation relation, by including $xy\rightsquigarrow_1 yx$ in the definition of $\rightsquigarrow$. (This is the approach we take with $\preccurlyeq$ in the proof below.) 

In the definition of $\rightsquigarrow$ the number of operations on the free algebra $(Q \cup R_k)^*$. In other words, it is not important that we took the free algebra over one binary operation, and we could take the free algebra with two binary and two nullary operations. This is what we do in the proof below and we denote the resulting computation relation by $\preccurlyeq$.

\subsection{Completeness of the reduction}
\label{subsec-completeness}

Let $M$ be the absolutely free $\{\circ, \1, \omt, \T \}$-algebra over $Q \cup R_k$.  We define $W:=M$, and $W'=S_M$, the set of all sections over $M$. For $u \in S_W$, we define $x \ldd u:=u(x \circ  \_)$, $u \rdd  y:=u(\_ \circ y)$, $x \ldo u:=u(x \omt  \_)$, $u \rdo  y:=u(\_ \omt y)$. If an element of $M$ has only $\circ$'s in it (for example it is a configuration) we usually write $\circ$ as $\cdot$ or simply as concatenation.

We define $(J, \jn)$ to be the free commutative semigroup over $M$; so, $M \subseteq J$ and the elements of $J$ admit a normal form up to commutativity: $x_1 \jn \ldots \jn x_m$, where $m \in \mathbb{Z}^+$ and $x_i \in M$. On $J$ we also define the operations $\circ, \1, \omt, \T$ by 
$(\bigvee x_i)\circ (\bigvee y_j)=\bigvee (x_i\circ y_j)$ and 
$(\bigvee x_i)\omt (\bigvee y_j)=\bigvee (x_i\omt y_j)$, for all $x_i, y_j \in M$, while the constants are the same as the ones of $M$; this yields the algebra $\m J:=(J, \circ, \1, \omt, \T, \jn)$.

We write $\preccurlyeq $ for the least $\{\circ, \omt, \jn\}$-compatible relation on $J$ that contains $\rightsquigarrow$, the semilattice axioms for $\omt$ (i.e.,  $x \omt y \preccurlyeq y \omt x$, $x \preccurlyeq x \omt x$, $x \omt y \preccurlyeq x$), $\circ$-associativity ($(x \circ y) \circ z \cong  x \circ (y \circ z)$), $\circ$-commutativity ($x \circ y \preccurlyeq  y \circ x$), 
and
the identity axioms ($x \circ \1 \cong  x \cong \1 \circ x$, $x \omt \T \cong  x$), where $x \cong y$  is short for ($x \preccurlyeq y$ and $y \preccurlyeq x$), for $x,y \in W$.
Finally, for $x \in W$ and $u \in S_W$, we write $x \mathrel{N} u$ iff $u(x) \preccurlyeq  q_f$.
We set $\m W_\texttt{M}:=(W, W', N, \circ, \1, \ldd, \rdd, \omt, \T, \ldo, \rdo )$.

\begin{lemma}
For every ACM $\texttt{M}$, $\m W_\texttt{M}$ is a commutative distributive residuated frame.
\end{lemma}
\begin{proof}
We first check the two nuclear properties. We have  $x \omt y \mathrel{N} u$ iff  $u(x\omt y) \preccurlyeq  q_f$ iff  $x  \mathrel{N} u(\_ \omt y)=u \rdo  y$. Likewise, for $\ldo$, $\ldd$, and $\rdd$.

Next, we check the four frame structural properties. If  $x \omt y \mathrel{N} u$ then $u(x\omt y) \preccurlyeq  q_f$ so   $u(y\omt x) \preccurlyeq u(x\omt y) \preccurlyeq  q_f$, hence $y \omt x \mathrel{N} u$. Likewise for exchange for $\circ$. If  $x\omt x \mathrel{N} u$, then  $u(x) \preccurlyeq  u(x\omt x) \preccurlyeq q_f$, so $x \mathrel{N} u$. If   $x \mathrel{N} u$, then  $u(x\omt y) \preccurlyeq  u(x) \preccurlyeq q_f$, so $x \omt y \mathrel{N} u$.
%
\end{proof}

\begin{lemma}\label{l: completness} If $\m W_\texttt{M} \models c\theta \leq q_f$ then $c \rightsquigarrow \overline{q_f}$, for every ACM $\texttt{M}$ and configuration $c$.
\end{lemma}

\begin{proof}
The inequality $c\theta \leq q_f$ holds in $\m W^+_\texttt{M}$ under the evaluation $e$ extending the assignment $e(x)=\gamma(\{x\})$ for $x \in Q \cup R_k$; here $\gamma(X)=X^{\btr \btl}$. So, $e(c\theta) \subseteq e(q_f)=\gamma(\{q_f\})$.
Since $\1 \circ q_f \preccurlyeq q_f$, we get $q_f \mathrel{N} u_\1$, where $u_\1:=(\1 \circ \_)$, i.e., $q_f \in u_\1^\btl$, so $\gamma(\{q_f\}) \subseteq  u_\1^\btl$. Thus, 
$e(c) \circ_\gamma e(\theta) = e(c\theta)\subseteq u_\1^\btl$.
We will first prove that $e(c)=\gamma(\{c\})$ and that $ 1\subseteq e(i)$, where $1:=1_{\m W^+_\texttt{M}}$ is the multiplicative unit of $\m W^+_\texttt{M}$.

 To prove $e(c)=\gamma(c)$, let $c=a_1\cdot a_2 \cdots a_m$, for some $a_i \in Q \cup R_k$. Then $e(c)=e(a_1\cdot a_2 \cdots a_m)=
 e(a_1)\circ_\gamma  e(a_2) \circ_\gamma \cdots \circ_\gamma e(a_m)=
 \gamma(a_1)\circ_\gamma  \gamma(a_2) \circ_\gamma \cdots \circ_\gamma \gamma(a_m)=
  \gamma( \gamma(a_1)\circ  \gamma(a_2) \circ \cdots \circ \gamma(a_m))=
  \gamma(a_1\cdot a_2 \cdots a_m)= 
 \gamma(c)$. We wrote $\gamma(x)$ for $\gamma(\{x\})$.

 To prove  $1 \subseteq e(i)$, note that: 
 $1 \subseteq e(i)=\bigcap (\gamma(d'_i) \rd \gamma(d_i))$ iff $1 \subseteq \gamma(d'_i) \rd \gamma(d_i)$ for all $d' \rd d \in I$ iff  $\gamma(d_i) \subseteq \gamma(d'_i)$ iff  $d \in \gamma(d')$  for all $d' \rd d \in I$. So, for a given $d' \rd d \in I$, we assume $\gamma(d') \subseteq u^\btl$, i.e., $d' \in u^\btl$, and will show $d \in u^\btl$; we distinguish cases, based on whether $d'$ is a join or not.
 If $d'$ is a configuration, then $d' \in u^\btl$ gives $u(d') \preccurlyeq \overline{q_f}$, while 
 $d' \rd d \in I$ gives $d \rightsquigarrow d'$, so $d \preccurlyeq d'$, hence  $u(d) \preccurlyeq u(d') \preccurlyeq  \overline{q_f}$; thus $d \in u^\btl$. 
 If $d'= q_1 \jn q_2$, then $d' \in u^\btl$ gives $u(q_1) \preccurlyeq \overline{q_f}$ and $u(q_2) \preccurlyeq \overline{q_f}$, while 
 $d' \rd d \in I$ gives $d \rightsquigarrow q_1 \jn q_2$, so $d \preccurlyeq q_1 \jn  q_2$, hence  $ u(d) \preccurlyeq u^{\m J}(q_1 \jn q_2)= u^{\m J}(q_1) \jn u^{\m J}(q_2) \preccurlyeq  \overline{q_f} \jn \overline{q_f} \preccurlyeq \overline{q_f}$; thus $d \in u^\btl$. 

Using $1 \subseteq e(i)$, we prove   $ e(\theta)=1$. Indeed, 
$e(t)=[e(i)\ld \gamma(q_f)] \ra \gamma(q_f)\supseteq 
[1\ld \gamma(q_f)] \ra \gamma(q_f)
=\gamma(q_f) \ra \gamma(q_f)=\top_{\m W^+_\texttt{M}}$, so
$1  \supseteq e(\theta)=e(\top) \ld e(t)\cap 1 \supseteq e(\top) \ld \top_{\m W^+_\texttt{M}} \cap 1  \supseteq 1$.

Now, from $e(c)=\gamma(c)$,  $e(\theta)=1$ and
$e(c) \circ_\gamma e(\theta) \subseteq u_\1^\btl$, we get  
$c \in \gamma(c)=e(c) \circ_\gamma e(\theta) \subseteq u_\1^\btl$, i.e., 
 $c \preccurlyeq  \1 \circ c=u_\1(c) \preccurlyeq  \overline{q_f}$,  hence $c \preccurlyeq \overline{q_f}$. We will show that  $c  \rightsquigarrow \overline{q_f}$. 

 We first prove the Claim that: if $z\preccurlyeq_{n+1} q_f$  and $z$ contains $\omt$, then there exists 
$z'$ such that $z\preccurlyeq_1 z'\preccurlyeq_n q_f$ and the first step was weakening. As a result, in a given computation witnessing $c \preccurlyeq \overline{q_f}$ there might be some instructional steps first, applied successively to $c$, and then the first application (if any) of contraction along a branch giving $c' \preccurlyeq c' \omt c'$, where $c'$ is a configuration, or more generally $c'(x) \preccurlyeq c'(x \omt x)$, where $x$ is a $\circ$-term. Since the structure $c' \omt c'$, or more generally the structure $c'(x \omt x)$, contains an $\omt$, by the Claim we may apply weakening  removing $\omt$ and getting back to $c'$, or more generally to $c'(x)$. Therefore, there is a computation where no contractions are applied (hence no steps involving $\omt$). Likewise we see that there a computation witnessing $c \preccurlyeq \overline{q_f}$ without $\1$ or $\T$ steps either. Hence, $c  \rightsquigarrow \overline{q_f}$.

To prove the Claim, we distinguish cases for what the first step in $z\preccurlyeq_{n+1} q_f$ is.
If $v(x) \preccurlyeq_1 v(x \omt x)\preccurlyeq_n q_f$, then by IH $v(x \omt x)\preccurlyeq_1 v(x)\preccurlyeq_{n-1} q_f$, in which case $v(x)\preccurlyeq_{n-1} q_f$, or 
$v(x \omt x)\preccurlyeq_1 v'(x \omt x) \preccurlyeq_{n-1} q_f$, in which case
$v(x) \preccurlyeq_1 v'(x) \preccurlyeq_1 v'(x \omt x) \preccurlyeq_{n-1} q_f$.
The remaining cases of non-instruction steps are similar.

To discuss the instruction steps we introduce some terminology. For $z \in M$ and  a leaf $d$ of the structure tree of $z$, we denote by $m(d)$ the biggest principal donwset of the structure tree of $z$ that contains $d$ and does not contain any $\omt$; also we denote by $\omt_d$ be the occurrence of $\omt$ right above $m(d)$, if any. 
We say that $z=v(m(d))$ is in \emph{$\circ$-form} (relative to $d$). In a step $u(x \omt y)\preccurlyeq_1 u(x)$ we say that the weakening was \emph{applied} to the displayed $\omt$ and that $y$ was \emph{removed}; if $v(y)=u(x \omt y)$, then we set $\overline{v}:=u(x)$. Therefore, if weakening was applied in $v(m(d))$, then either it was applied to $\omt_d$ and $m(d)$ was removed, i.e., $v(m(d)) \preccurlyeq \overline{v}$, or it was applied to some $\omt$ inside $v$ (possibly to $\omt_d$ even) but $m(d)$ was not removed, i.e., $v(m(d)) \preccurlyeq v'(m(d))$, for some $v'$.

If $v(m(q)) \preccurlyeq_{1} v(m(q_1)) \jn v(m(q_2))\preccurlyeq_n q_f$, where the expressions are in $\circ$-form, then by IH either $v(m(q_i)) \preccurlyeq_1 v'(m(q_i)) \preccurlyeq_{n-1} q_f$ for both $i \in \{1,2\}$  (neither one of $m(q_1)$ and $m(q_2)$ was removed by the weakening), in which case 
$v(m(q)) \preccurlyeq_{1} v'(m(q)) \preccurlyeq_1 v'(m(q_1)) \jn v'(m(q_2))\preccurlyeq_{n-1} q_f$, or  
$v(m(q_i)) \preccurlyeq_1 \overline{v}\preccurlyeq_{n-1} q_f$ for some $i \in \{1,2\}$ (at least one of $m(q_1)$ and $m(q_2)$ was removed by the weakening), in which case we have  $v(m(q)) \preccurlyeq_{1} \overline{v}\preccurlyeq_{n-1} q_f$.

If $v(m(qr)) \preccurlyeq_{1} v(m(q'))\preccurlyeq_n q_f$, where the expressions are in $\circ$-form, then by IH either $v(m(q')) \preccurlyeq_1 v'(m(q')) \preccurlyeq_{n-1} q_f$, in which case 
$v(m(qr)) \preccurlyeq_{1} v'(m(qr)) \preccurlyeq_1 v'(m(q')) \preccurlyeq_{n-1} q_f$, or  
$v(m(q')) \preccurlyeq_1 \overline{v}\preccurlyeq_{n-1} q_f$, in which case we have  $v(m(q)) \preccurlyeq_{1} \overline{v}\preccurlyeq_{n-1} q_f$.
 \end{proof}

  \begin{corollary}
If $\mathcal{V}$ is variety of GBI algebras  that contains $\m W_\texttt{M}$, for some  ACM $\texttt{M}$, then $\mathcal{V} \models c\theta \leq q_f$ iff $c \rightsquigarrow \overline{q_f}$, for every configuration $c$ of  $\texttt{M}$.
\end{corollary}

\begin{proof}
If $c \rightsquigarrow \overline{q_f}$, then
$\mathsf{GBI} \models c\theta \leq q_f$, by Lemma~\ref{l: soundness}, so 
$\mathcal{V} \models c\theta \leq q_f$. 
Conversely, if 
$\mathcal{V} \models c\theta \leq q_f$, then
$\m W_\texttt{M} \models c\theta \leq q_f$, so $c \rightsquigarrow \overline{q_f}$,  by Lemma~\ref{l: completness}. 
\end{proof}

 Recall that there are ACMs  with undecidable acceptance problem.
 
 \begin{corollary}\label{c: UndInt}
If a variety of GBI algebras  contains $\m W_\texttt{M}$, for some undecidable ACM $\texttt{M}$, then its equational theory is undecidable.    
\end{corollary}

 \begin{corollary}\label{c: (G)BI}
The equational theory of BI is undecidable. The same holds for GBI.   
\end{corollary}

Note that the argument via residuated frames establishes the undecidability in a way that does not use the axiom of choice either. Also, note that it establishes undecidability for fragments that do not contain the bottom element $\bot$ nor the negation connective $\neg$.

\end{document}